\documentclass[final,leqno,onefignum,onetabnum]{siamltex1213}
\usepackage{graphicx}
\usepackage{epsfig} 
\usepackage{microtype}
\usepackage{amsmath} 
\usepackage{amssymb}  
\usepackage{amscd}
\usepackage{float}
\usepackage{cite}
\usepackage{color}

\newtheorem{Definition}{Definition}
\newtheorem{Remark}{Remark}
\newtheorem{pro}{Proposition}
\newtheorem{lem}{Lemma}
\newtheorem{cor}{Corollary}

\newcommand{\G}{\mathbb{G}}
\newcommand{\R}{\mathbb{R}}
\newcommand{\ol}{\overline}
\renewcommand{\cal}{\mathcal}
\renewcommand{\(}{\left (}
\renewcommand{\)}{\right )}
\renewcommand{\diag}{\operatorname{diag}}
\newcommand{\rk}{\operatorname{rank}}

\newcommand{\SE}{{\rm SE}}
\newcommand{\wh}{\widehat}
\newcommand{\qed}{\hfill\blacksquare}

\title{Global Stabilization of Triangulated Formations}

\author{Xudong Chen, 
Mohamed-Ali Belabbas, 
Tamer Ba\c sar\thanks{X. Chen, M.-A. Belabbas, T. Ba\c sar are with the Coordinated Science Laboratory, University of Illinois at Urbana-Champaign, emails: \{xdchen, belabbas, basar1\}@illinois.edu}. Corresponding author: X. Chen. 
T. Ba\c sar was partly supported by the U.S. Air Force Office of Scientific Research (AFOSR) MURI grant FA9550-10-1-0573; M.-A. Belabbas was partly supported by NSF ECCS 13-07791 and NSF ECCS CAREER 13-51586; X. Chen was supported jointly by (AFOSR) MURI grant FA9550-10-1-0573 and by NSF ECCS CAREER 13-51586.
}
\begin{document}

\maketitle

\begin{abstract} Formation control deals with the design of decentralized control laws that stabilize mobile, autonomous agents at prescribed distances from each other. We call any configuration of the agents a target configuration if it  satisfies the inter-agent distance conditions. It is well-known that when the distance conditions are defined by a rigid graph, there is a finite number of target configurations modulo rotations and translations of the entire formation. We can thus recast the objective of formation control as stabilizing one or many of the target configurations. A major issue is that such control laws will also have  equilibria corresponding to configurations which do not meet the desired inter-agent distance conditions; we refer to these as undesirable configurations. The undesirable configurations become problematic if they are also stable. Designing decentralized control laws  whose  stable equilibria are all target configurations in the case of a general rigid graph is still an open problem. We provide here a new point of view on this problem, and propose a partial solution by exhibiting  a  class of rigid graphs and  control laws for which all stable equilibria are  target configurations. 
\end{abstract}



\section{Introduction}

The design of control laws stabilizing a group of mobile autonomous agents has raised a number of issues related to the number of  equilibria, their asymptotic stability properties and the level of decentralization  of the system. In rigidity-based, or distance-constrained, formation control, one assigns agents to the vertices of a  rigid graph 
and specifies the \emph{target distances} between the pairs of agents linked by  edges.  We refer to any configuration of the agents that satisfies these distance requirements as a target configuration. The rigidity of the graph thus ensures that there is a finite number of target configurations up to rotations and translations of the plane.  The goal of a  decentralized formation control law in this context is  to either locally or globally stabilize a subset of the target configurations. However, the decentralization constraints and geometry of the state-space make the appearance of ancillary, undesirable equilibria inevitable~\cite{AB2013TAC}. Thus, global stabilization of formation systems is obviously not achievable. However, if the undesired equilibria are all unstable, global stabilization is \emph{practically} achieved. We thus call a control law {\bf essentially stabilizing} if it stabilizes {\em only} target configurations. 

The literature on formation control has been enriched in the past few years. We mention some relevant results to our study here, and refer to references in these papers for a more extensive list. The relationship between the level of decentralization of the formation and the existence of essentially stabilizing control laws has been studied in~\cite{AB2013TAC}. There, it was shown that a particular information flow in a formation control system implied the existence of undesirable yet stable equilibria, regardless of the control used. In~\cite{AL2014ECC}, it was shown that one could not locally stabilize all target configurations for a class of directed formations.  Amongst positive results, it has been shown in~\cite{anderson2007control} that the triangle formation is essentially stabilizable; in~\cite{dimarogonas2008stability, dimarogonas2009further} that if the formation graph is a {\it tree} (though not rigid), then the target configurations are essentially stabilizable;    
and in~\cite{cao2008control,guo2010adaptive} that a class of \emph{acyclic directed} formations is similarly essentially stabilizable. We further refer to~\cite{chen2015decentralized,tian2013global,lin2014distributed} for other control models which have also addressed the issue of global convergence towards the target configurations. Finally, we mention that there are other notions of formation control, where the desired equilibrium configurations are  characterized not by target distances, but by bearings. This is referred to as \emph{bearing-constrained} formation control, as opposed to distance-constrained formation control. We refer the reader to~\cite{zhao2015bearing} for more information.

From these papers, we conclude that in general, the problem of characterizing the set of essentially stabilizing control laws is challenging; indeed, describing the  rigid graphs for which there exists an essentially stabilizing control law remains an open problem. The contribution of this paper is  to exhibit a class of \emph{undirected graphs}, termed {\it triangulated Laman graphs}, and an associated class of essentially stabilizing control laws, for which all stable equilibria are target configurations.  Moreover, we provide a new approach for investigating this type of problems, which might be of independent interest.  Our main result is to provide a formula for computing the signature of the Hessian of a class of potential functions for formation control. This formula allows us to consider subsystems of the formation independently, and overall makes the evaluation of the signature tractable. 

\paragraph{The model} We now describe the model under study in this paper in precise terms. Let $G=(V,E)$ be an undirected graph with vertex set $V:=\{1,\cdots,n\}$  and edge set $E$. Two vertices are said to be {\it neighbors} if there is an edge joining them. We denote by $\mathcal{N}_i$ the set of neighbors of vertex $i$.  Let $ x_i \in \R^2$, $i=1,\ldots, n$, be the coordinate of agent $i$. With a slight abuse of notation, we refer to agent $i$ as $x_i$. For every edge $(i,j) \in E$, we let $d_{ij}:= \| x_i- x_j\|$ be the Euclidean distance between agents $x_i$ and $x_j$, and denote by $\ol d_{ij}$ the corresponding {\bf target distance}. 
The equations of motion of the $n $ agents $ x_1,\cdots, x_ n$ in $\mathbb{R}^2$ are given by       
\begin{equation}\label{MODEL}
\dot{ x}_{i} = \sum_{j\in \mathcal{N}_i}u_{ij}\(d_{ij},\ol d_{ij}\) ( x_j- x_i),\hspace{10pt} \forall i \in V.
\end{equation}
Each function $u_{ij}(d_{ij},\ol d_{ij})$ is assumed to be continuously differentiable in both arguments. For a fixed $\ol d_{ij}>0$, the function $u_{ij}(d_{ij}, \ol d_{ij})$ takes the distance $d_{ij}$ as the state feedback. 
Often, the feedback control laws $u_{ij}$ are designed so that each $u_{ij}(\cdot, \ol d_{ij})$ has a unique zero at $\ol d_{ij}$, i.e., $u_{ij}(\ol d_{ij},\ol d_{ij}) = 0$. 
In other words, if all pairs of agents $ x_i$ and $ x_j$, with $(i,j)\in E$, reach their target distances $\ol d_{ij}$ (or equivalently, the agents reach a target configuration),  then the entire formation is at an equilibrium. However, the converse may not be true, i.e., given a rigid graph $G$, and a set of control laws $u(\cdot, \ol d_{ij})$, there may exist stable equilibria which are not target configurations; indeed, for a rigid graph $G$, it is not even clear whether there exist such control laws $u_{ij}(\cdot,\ol d_{ij})$ stabilizing only the target configurations.


It is known that system~\eqref{MODEL} is a gradient dynamics (an associated potential function will be given in Section~\S\ref{sec:section2}). We can thus rephrase our goal of obtaining an essentially stabilizing control law as one of designing a potential function whose local minima are all target configurations. 
We note here that this class of gradient formation control systems has also been investigated  from many other perspectives.  Questions concerning the level of interaction laws for organizing such systems~\cite{dimarogonas2008stability,dimarogonas2009further,krick2009, GP, XC2014ACC, chen2016swarm}, questions about system convergence~\cite{dimarogonas2008stability,dimarogonas2009further,XC2014ACC,chen2016swarm}, and questions about local stability~\cite{krick2009} and robustness~\cite{USZB,mou2014CDC} have all been treated to some degree for the case of gradient dynamics. 
These are in general difficult questions. For example, even counting the number of critical formations in system~\eqref{MODEL} is hard. A partial solution to this counting problem was provided in~\cite{BDO2014CT} in the one-dimensional case. We also refer to~\cite{UH2013E} for using Morse theoretic ideas for this purpose. 

The remainder of the paper is organized as follows. In Section~\S\ref{sec:section2}, we recall some known facts on convergence of the system and associated potential functions. We also introduce the class of triangulated Laman graphs and an associated class of triangulated formation systems. We then state the main result of the paper. Specifically, the main theorem  characterizes a set of essentially stabilizing control laws for the formation control system.  Sections~\S\ref{MICF} and~\S\ref{ssec:morsebott} are devoted to establishing the properties that are needed for proving the main theorem. A detailed organization of these two sections will be described following the statement of the main theorem. 
 We provide conclusions in Section~\S\ref{sec:conclusion}. The paper ends with an Appendix which presents a proof for one of the supporting propositions.


\section{Preliminaries and the main theorem}\label{sec:section2}

\subsection{The control laws and the system convergence}
Let $G=(V,E)$ be an undirected graph of~$n$ vertices. We define the  {\it configuration space}  $P_{G}$ associated with $G$ as follows:
\begin{equation*}
P_{G}:=\left\{( x_1,\cdots, x _n)\in \mathbb{R}^{2n} \mid  x_i\neq  x_j, \, \forall (i,j)\in E \right\}. 
\end{equation*} 
Equivalently, $P_{G}$ is the set of embeddings of the graph $G$ in $\mathbb{R}^2$ for which  vertices linked by an edge have distinct positions. We call a pair $(G,p)$, with $p\in P_G$, a {\bf framework}. Of course, the order in which the coordinates of the vertices appear in the vector $p$ is not important. Hence, we  use the notation $p|_i$, for $i\in V$, to refer to the coordinates of the vertex $i$ in the embedding of $G$ given by $p$. Let $G' = \(V',E' \)$ be a sub-graph of $G$. We call an element $p' \in P_{G'}$ a {\bf sub-configuration} of an element $p \in P_G$ induced by $G'$ if $p' |_i=p |_i$ for all $i \in V'$.

We now formalize the type of systems introduced in~\eqref{MODEL} as follows:

\begin{Definition}[Graph induced systems]\label{def:inducedsubsystem}  
Let $G = (V,E)$ be an undirected graph. We call system~\eqref{MODEL} a formation control system {\bf induced by $G$}, with control laws $u_{ij}$ for $(i,j)\in E$. Let $G'=  (V', E')$ be a subgraph of $G$. A formation control system induced by $G'$ with control laws $u'_{ij}$, for $(i,j)\in E'$,  is a {\bf subsystem} of~\eqref{MODEL}  if $u'_{ij} = u_{ij}$ for all $(i,j) \in E'$.
\end{Definition}

An important property of the class of systems introduced in Definition~\ref{def:inducedsubsystem} is that their dynamics can be written as gradients of real-valued functions. The associated potential function $\Phi: P_G\longrightarrow \R$ is given by 
\begin{equation}\label{PHI}
\Phi(p) := \sum_{(i,j)\in E}\displaystyle\int^{\|x_i-x_j\|}_{1} s\, u_{ij}(s,\ol d_{ij})ds.
\end{equation}
A direct computation shows that $\dot x_i =-\partial \Phi/{\partial x_i}$ for all $i\in V$. Note that if
 $G'= (V', E')$ is a subgraph of $G$, then the subsystem of~\eqref{MODEL} induced by $G'$ is also a gradient system. The corresponding potential function $\Phi': P_{G'} \longrightarrow \R$ takes the same expression as~\eqref{PHI}, but this time with the summation over $E'$.   
We call $\Phi'$ the potential function {\it induced by} $G'$.



\paragraph{Monotone attraction/repulsion function} We  now introduce the class of control laws $u_{ij}$ studied in this paper. 
 A typical example of such a control law is 
\begin{equation*}
u_{ij}\(d,\,  \ol d_{ij}\) =  1 - \(\ol d_{ij}/d\)^2. 
\end{equation*}
This law is similar to the gradient control law in~\cite{krick2009} scaled by $1/d^2$. 
Under this control law, two agents $x_i$ and $x_j$ attract (resp. repel) each other if their mutual distance is larger (resp. smaller) than the target distance $\ol d_{ij}$. 
%
We will require a little more for the class of control laws considered in this paper. 
Let $\mathbb{R}_+$ be the set of strictly positive real numbers, and let $\operatorname{C}^1(\mathbb{R}_+,\mathbb{R})$  be the set of continuously differentiable functions from $\mathbb{R}_+$ to $\mathbb{R}$. We have the following definition:

\begin{Definition}[Monotone att./rep. function]\label{def:interactionfunction} We call $f \in \operatorname{C}^1(\R_+,\R)$ a {\bf monotone attraction/repulsion function} if it satisfies the following two conditions:

\begin{enumerate}
\item[C1.]  For any  $x>0$, we have  
$
d(xf(x))/dx>0
$, 
and $f(x)$ has a unique zero. 
\item[C2.]  $\lim_{x\to 0+}\int^1_{x} s f(s)ds=-\infty$. 
\end{enumerate}\,
\end{Definition}
%

Note that if $u_{ij}(\cdot,\ol d_{ij})$ is a monotone att./rep. function, with the unique zero at $\ol d_{ij}$, then, from condition~C1, $u_{ij}(d,\ol d_{ij})<0$ if  $d<\ol d_{ij}$ and $u_{ij}(d,\ol d_{ij})> 0$ if $d> \ol d_{ij}$. 
We also note that condition~C1 implies that if the graph $G$ is connected, then there is no escape of any agent along the evolution: On the one hand, for an edge $(i,j) \in E$, we have
$
\lim_{x\to\infty}\int^{x}_{1} su_{ij}(s,\ol d_{ij}) ds = \infty 
$. 
On the other hand,  the dynamics~\eqref{MODEL} is the gradient flow of $\Phi$, and hence $\Phi$  decreases along trajectories. We thus conclude that any pair of adjacent agents cannot be too far away from each other along the evolution. By the same reason, we note that condition~C2 prevents collisions of adjacent agents along the evolution because C2 implies that $\Phi$ is infinite when the distance of separation between two adjacent agents is zero. 
In fact, we can show that the solutions of system~\eqref{MODEL} exist for all time, and converge to the set of equilibria:


\begin{lem}\label{CONVGG}
If the control laws $u_{ij}(\cdot, \ol d_{ij})$'s are monotone att./rep. functions with $u_{ij}(\ol d_{ij},\ol d_{ij}) = 0$,  then for any initial condition $p(0)\in P_{G}$, the solution $p(t)$ of system \eqref{MODEL} exists for all $t \geq 0$ and converges to the set of equilibria. 
\end{lem}

 We note that Lemma~\ref{CONVGG} holds even if the control laws have fading attractions, i.e., $\lim_{x\to\infty} du_{ij}(d,\ol d_{ij}) = 0$, for all $(i,j) \in E$. We refer to~\cite{chen2016swarm} for more details, including a proof of the lemma above.

\subsection{The potential function and its invariance}\label{ssec:potentialinvariance}

We now derive  some properties of the potential function $\Phi$. Note that $\Phi$ depends only on relative distances between the agents, thus it is invariant if we translate and/or rotate the entire configuration in $\R^2$. We  now describe this property in precise terms.

The special Euclidean group $\SE(2)$ has a natural action  on the configuration space. Recall that  $\gamma$ in $\SE(2)$ can be represented by a pair $(\theta,  v)$ with $\theta$ a rotation matrix, 
and $ v$ a vector in $\mathbb{R}^2$.  With this representation, the  multiplication of two elements $\gamma_1 = (\theta_1, v_1)$ and $\gamma_2 = (\theta_2, v_2)$ of $\SE(2)$ is given by 
$
\gamma_2\cdot \gamma_1  = (\theta_2 \theta_1, \theta_2  v_1 +  v_2) 
$.  
The action of $\SE(2)$  on $P_{G}$ alluded to above is defined as follows: given  $\gamma = (\theta, v)$ in $\SE(2)$ and  $p=( x_1,\cdots,  x _n)$ in $P_{G}$, we let
\begin{equation}\label{def:groupaction}
\gamma\cdot p:=(\theta  x_1+ v,\cdots,\theta  x _n+ v).
\end{equation} 
Now, let 
$
O_p:= \{ \gamma \cdot p  \mid \gamma \in \SE(2)\} 
$. We call $O_p$ the orbit of $\SE(2)$ through  $p \in P_{G}$. 
Then, the potential $\Phi$  keeps the same value over an orbit, i.e., 
$\Phi(p) = \Phi(\gamma\cdot p)$  for all $\gamma \in \SE(2)$  and  for all $p\in P_G$.

To proceed, we let $\nabla \Phi$ be the gradient of $\Phi$, i.e., $\nabla \Phi(p):= \partial \Phi(p) / \partial p$. We call $p \in P_G$ a {\bf critical point} of $\Phi$, or equivalently an {\bf equilibrium} of the gradient system  
$\dot p = - \nabla \Phi(p)$, 
if $\nabla \Phi(p) = 0$.  The invariance property of $\Phi$ over an orbit implies that 
\begin{equation*}
\nabla\Phi(\gamma\cdot p) = \operatorname{diag}(\theta,\cdots, \theta) \nabla\Phi(p),
\end{equation*}
where  $\operatorname{diag}(\theta,\cdots,\theta)$ is a block-diagonal matrix with $n $ copies of $\theta$.  Since $\operatorname{diag}(\theta,\cdots, \theta)$ is invertible, if $p$ is a critical point of~$\Phi$, then so is $p'$ in $O_p$.  We  thus refer to the orbit $O_p$  as a {\bf critical orbit} if $\nabla \Phi(p)=0$. 
Now, let $H_p$ be the {\it Hessian} of $\Phi$ at $p$, i.e., 
$
H_p:= \partial^2 \Phi(p) / \partial p^2 
$.  
The following lemma presents well-known facts about the Hessian matrix of an invariant function:

\begin{lem}\label{LEQUIV}
Let $\Phi$ be  a function invariant  under a Lie-group action over a Euclidean space. Let $k$ be the dimension of a critical orbit $O_p$ under the group action and denote by  $H_{p}$ the Hessian of $\Phi$ at $p$.  Then, for any two configurations $p_1$ and $p_2$ in $O_p$, the two matrices $H_{p_1}$ and $H_{p_2}$ are similar.  In addition, the Hessian $H_{p}$  has at least $k$ zero eigenvalues. The null space of $H_p$ contains  at least the tangent space of $O_p$ at $p$. 
\end{lem}

We refer to~\cite{field1980equivariant} for more facts on equivariant dynamical systems. 
In our case, if $p' = \gamma \cdot p$ for $\gamma = \(\theta, v\)$, then 
$$
H_{p'} = \diag(\theta, \ldots, \theta)\, H_{p} \,\diag(\theta^\top,\ldots, \theta^\top ),
$$ which shows that indeed, $H_{p'}$ is similar to $H_p$.
In the case of the group $\SE(2)$, the  orbits $O_p$, for $p\in P_{G}$, are of dimension~$3$.  From Lemma~\ref{LEQUIV}, there are thus at least three zero-eigenvalues of $H_p$.  Moreover, the eigenvalues of $H_p$, for any $p \in O_p$, are the same. The following definition is thus well posed:

\begin{Definition}[Nondegenerate critical orbits]\label{def:critstaborbits}
A critical orbit $O_p$ is  {\bf nondegenerate} if there are exactly three zero eigenvalues of $H_p$. Furthermore, it is   {\bf exponentially stable} if all the other eigenvalues are positive. \end{Definition}

We further need the following definition about the potential function $\Phi$:

\begin{Definition}[Equivariant Morse functions]
A potential function $\Phi$ is said to be an {\bf equivariant Morse function} if there are only \emph{finitely} many critical orbits of $\Phi$, and moreover, each critical orbit  is nondegenerate.   
\end{Definition}

\subsection{Triangulated formations and the main result}
Let $G=(V,E)$ be an undirected graph. Let the distance function $\rho_{G}: P_{G}\longrightarrow \mathbb{R}^{|E|}_+$ be defined by  
\begin{equation}\label{maprho}
\rho_{G}(p) := \(\cdots, \| x_i -  x_j\|^2,\cdots\)_{(i,j)\in E} 
\end{equation} 
where the ordering of the edges in $E$ is irrelevant.
The graph $G$ is called {\bf rigid} in $\mathbb{R}^2$ if for any vector ${\bf d} := (\cdots, d_{ij},\cdots) \in \mathbb{R}^{|E|}_+$, the preimage $\rho^{-1}({\bf d})$ is comprised of {\em finitely many} orbits in $P_G$.  
The graph $G$ is called {\it minimally rigid} if it is not rigid  after taking out any of its edges~\cite{laman1970}. A {\bf Laman graph} is a minimal rigid graph in $\R^2$.  
 
It is well-known that every Laman graph can be obtained via a so-called \emph{Henneberg sequence}; a Henneberg sequence $\{G(l)\}_{l\ge 1}$ is a sequence of minimally rigid graphs obtained via two basic operations: edge-split and vertex-add. 
We refer to~\cite{graver1993combinatorial} for more details about these operations. 
We define {\bf triangulated Laman graphs} as those graphs obtained by imposing constraints on the type of operations allowed in a Henneberg sequence. Precisely, we have the following definition: 

\begin{Definition}[Triangulated Laman graphs]\label{def:triangLaman}
A graph $G$ is a {\bf triangulated Laman graph} if it is an element of any sequence of graphs $\{G(l)\}_{l \geq 1}$ obtained as follows:  Start with $G(1)$, the graph comprised of a single vertex, and then $G(2)$, the graph obtained by linking a new vertex to the vertex in $G(1)$ via an edge.    
For $l \geq 3$, the graph $G(l)$ is obtained from $G(l-1)$ by linking  a  new vertex to two adjacent vertices in $G(l - 1)$ via two new edges.
\end{Definition}

 In other words, only the operation of vertex-add is allowed in the Henneberg sequence, and in addition, the new vertex, appearing in $G(l+1)$ for $l \ge 2$, cannot be adjacent to two arbitrary vertices, but rather to two vertices connected by an existing edge in $G(l)$. See Fig.~\ref{Hcon} for an illustration. Hence, a Henneberg sequence 
$\G =  \{G(l)\}^{n}_{l = 1}$,  with $G(l)=(V(l),E(l))$, 
for a triangulated Laman graph $G$ of $n $ vertices, satisfies the condition that each $G(l)$ is a subgraph of $G(l+1)$. Moreover, the cardinalities of $V(l)$ and $E(l)$, for $l \ge 2$, satisfy the condition that $|V(l)|= l$ and  $|E(l)| = 2l - 3$.

We also note that if $G$ is a triangulated Laman graph of more than two vertices, then there is at least a vertex of {\it degree}~$2$, i.e., a vertex adjacent to two and only two vertices. Indeed, by following a Henneberg sequence of $G$, we have that the last vertex appearing in the sequence is of degree~$2$. 

\begin{figure}[h]
\begin{center}
\includegraphics[width=0.6\textwidth]{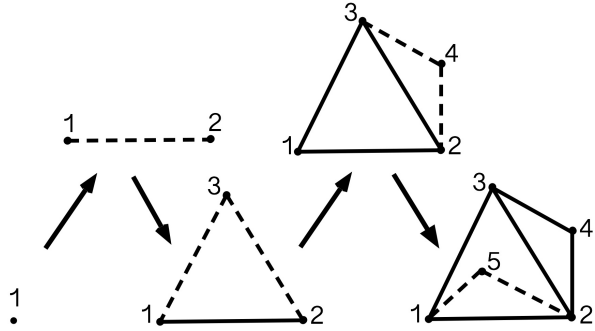}
\caption {An example of a triangulated Laman graph. Starting with vertex~$1$, and edge $(1,2)$, we then subsequently join vertices 3, 4 and 5 to two existing adjacent vertices.}
\label{Hcon}
\end{center}
\end{figure}

Let $G$ be a  triangulated Laman graph. We say that a subgraph $G'$  of $G$ is a {\it $3$-cycle} if $G'$ is a complete graph on three vertices. 
We further introduce the following definition: 

\begin{Definition}[Strongly rigid frameworks]
A framework $(G,p)$ is {\bf strongly rigid} (or simply $p$ is strongly rigid) if $p$ satisfies the following condition: if vertices $i$, $j$ and $k$ of $G$ form a $3$-cycle of $G$, then the triangle formed by agents $ x_i$, $ x_j$ and $ x_k$ is nondegenerate, i.e., $x_j- x_i$ and $x_k - x_i$ are linearly independent.  
\end{Definition}

Note that if $p$ is strongly rigid, then so is any $p'\in O_p$. Hence, there is no ambiguity in saying that an orbit $O_p$ is strongly rigid.  
 Let $\rho_G$ be the distance map defined in~\eqref{maprho}. A framework $(G,p)$ is said to be {\it infinitesimally rigid}~\cite{graver1993combinatorial} (or simply, $p$ is infinitesimally rigid) if the Jacobian matrix $\partial \rho_G(p)/ \partial p$ is of rank $2n - 3$. 
We have  the following fact:

\begin{pro}\label{pro:srcopendense}
Strongly rigid configurations are infinitesimally rigid, Moreover, they form an open dense subset of $P_{G}$.  
\end{pro}

We refer to Appendix~A.1 for a proof of Proposition~\ref{pro:srcopendense}.

Let $p = (x_1,\ldots, x _n)$ be a strongly rigid configuration, and recall that $d_{ij}$ is the distance between $x_i$ and $x_j$. 
Note that if vertices $i$, $j$ and $k$ form a $3$-cycle of $G$, then
\begin{equation*}\label{eq:eatingaclementinewithXudong}
\left\{
\begin{array}{l}
d_{ij} + d_{ik} > d_{jk},\\
d_{ij} + d_{jk} > d_{ik},\\
d_{ik} + d_{jk} > d_{ij}.
\end{array}\right.
\end{equation*}
We then say that the set of distances $\{d_{ij} \mid  (i,j)\in E  \}$ satisfies the {\bf strict triangle inequalities} associated with $G$. 
Conversely, if the target distances $\ol d_{ij}$, for $(i,j)\in E$, satisfy the strict triangle inequalities, then there will be \emph{strongly rigid} configurations $p = (x_1,\ldots, x_n)$, with $d_{ij} = \ol d_{ij}$ for all $(i,j)\in E$. Indeed, by the construction introduced in Definition~\ref{def:triangLaman}, there are $2^{n-2}$ orbits comprised of these configurations. 

We note here that this triangulated formation structure has also been considered in a class of {\it acyclic directed} formation control models (see, for example,~\cite{JB2006cdc,cao2008control,guo2010adaptive}), and the authors there have also addressed the issue of global convergence of target configurations. However, the analysis there can not be applied here; indeed, in an acyclic directed formation control model, there is a {\it leader-follower structure} which implies that the dynamics of the system has a triangular structure, i.e., the dynamics of the leader feeds into the dynamics of the follower but not the other way around. Yet, in an undirected formation control model, there does not exist such a leader-follower structure, and hence system~\eqref{MODEL} does not have a triangular structure.    

We now introduce the class of formation systems addressed in this paper.
\begin{Definition}[Triangulated formation systems]\label{def:triangsys}
Let $G=(V,E)$ be a triangulated Laman graph and let the target distances $\ol d_{ij}$, for $(i,j)\in E$,  satisfy the strict triangle inequalities associated with $G$. Let $u_{ij}(\cdot, \ol d_{ij})$, for $(i, j)\in E$, be monotone att./rep. functions, with $u_{ij}(d_{ij},\ol d_{ij}) = 0$. Further, we let the associated potential function $\Phi$, defined by~\eqref{PHI}, be an equivariant Morse function. Then, a formation system induced by $G$, with the control laws $u_{ij}$'s, is a {\bf triangulated formation system}.
\end{Definition}

\begin{Remark}\normalfont In Definition~\ref{def:triangsys}, there is a technical condition that the potential function $\Phi$ is an equivariant Morse function.  We note here that if the graph $G$ is a triangulated Laman graph as the case here, then $\Phi$ is  {\em generically} an equivariant Morse function. We refer to~\cite{chen2015reciprocal} for a proof of this fact. Yet, whether the result holds for $G$ an arbitrary rigid graph remains an unanswered question.  $\qed$
\end{Remark}

With the preliminaries above, we state the main result of the paper:

\begin{theorem}\label{MAIN}
Let system~\eqref{MODEL} be a triangulated formation system. Then, the following hold:
\begin{enumerate}
\item A critical orbit $O_p$ of $\Phi$ is exponentially stable if and only if  it is strongly rigid. There are $2^{n-2}$ stable critical orbits each of which satisfies the condition that $d_{ij}=\ol d_{ij}$ for all $(i,j)\in E$.
\item For almost all initial conditions $p(0)\in P_{G}$, the trajectory $p(t)$, generated by system~\eqref{MODEL}, converges to one of the $2^{n-2}$ stable critical orbits. 
\end{enumerate}\,
\end{theorem}
\begin{Remark}
\normalfont 
We characterize here precisely the set of initial conditions for which the solution of~\eqref{MODEL} converges to a stable critical orbit. First, note that there exist critical orbits of $\Phi$ other than the orbits of target configurations. To see this, consider for example a subset of $P_G$ comprised of the {\em line configurations}---these are the configurations $p = (x_1,\ldots, x_n)$ for which $x_1,\ldots, x_n$ are aligned. It is known that such a set is positive invariant under the dynamics~\eqref{MODEL} (see, for example,~\cite{chencontrollability,sun2015rigid}). Since, from Lemma~\eqref{CONVGG}, a solution of~\eqref{MODEL} always converges to the set of equilibria, there exists at least a critical orbit of $\Phi$ comprised only of line configurations. But, from Theorem~\ref{MAIN}, any such critical orbit is unstable. It is in general hard to locate, or even just count the number of, all the unstable critical orbits. However, since the potential $\Phi$ is (generically) an equivariant  Morse function for $G$ a triangulated Laman graph, there are only finitely many unstable critical orbits. We label them as $O_{q_1},\ldots, O_{q_l}$. Now, let $\phi_t(p)$ be the solution of~\eqref{MODEL} at time~$t$ with initial condition $p$.  For each $i = 1,\ldots, l$, the so-called {\bf stable manifold} of $O_{q_i}$ is defined as follows:
$$
W^s(O_{q_i}):=\{ p\in P_G \mid \lim_{t\to \infty}\phi_t(p) \in O_{q_i} \}. 
$$
Because $O_{q_i}$ is unstable, the codimension of $W^s(O_{q_i})$ is at least~$1$, i.e., $\dim P_G - \dim W^s(O_{q_i}) \ge 1$. Further, we let
$$
Q := \cup^l_{i  = 1} W^s(O_{q_i}).
$$
Then, by construction, the set $Q^c:=P_G - Q$, defined as the complement of $Q$ in $P_G$, is open and dense in $P_G$. Moreover, for any initial condition $p(0)\in Q^c$, the solution of~\eqref{MODEL} converges to one of the stable critical orbits.  $\qed$
\end{Remark}
\

The implication of Theorem~\ref{MAIN} is that the control laws considered in this paper are essentially stabilizing. The next two sections are devoted to proving Theorem~\ref{MAIN}. In Section~\S\ref{MICF}, we introduce an edge-set partition (called the independence partition), which decomposes a framework $(G,p)$ into sub-frameworks $\{(G_i,p_i)\}^m_{i=1}$ with each $G_i$ a triangulated Laman graph and $p_i$ a line sub-configuration. We also describe relevant properties associated with the edge-set partition. In Section~\S\ref{ssec:morsebott}, we compute the Hessian of $\Phi$ at critical points. In particular,  we provide a formula for computing the signature of the Hessian, which might be of independent interest.

\section{Independent partitions}\label{MICF}

In this section, we introduce the notion of an {\it independent partition} for a framework $(G,p)$.  This is a partition of the \emph{edge-set} of $G$ such that, roughly speaking, \emph{edges that are aligned (in the embedding of $G$ by $p$)  all belong to the same subset}. Specifically, the independent partition for a framework $(G,p)$ can be defined via a Henneberg sequence of $G$, and we give below a precise definition: 

\begin{Definition}[Independent partition]\label{def:indpar}  Let $G = (V, E)$ be a triangulated Laman graph of $n $ vertices, for $n \ge 2$.  Let $(G,p)$ be a framework, for $p\in P_G$. 
Let $\G = \{G(l) \}^{n }_{l = 1}$, with $G(l)= (V(l), E(l))$, be a Henneberg sequence of $G$. 
We label the vertices of $G$ with respect to the order in which they appear in the sequence.
An {\bf independent partition} for $(G, p)$ is a partition of $E$, which can be defined  recursively along the sequence $\G$ as follows:
\begin{enumerate}
\item {\it Initial step.} Start with $E(2)$, which contains only one edge $(1,2)$, and hence the partition is trivial.
 
\item {\it Inductive step.} Assuming that $E(k-1)$, for $k \ge 3$, has been partitioned as 
$E(k-1) = \sqcup^m_{i=1}E_i$, 
we define the partition of $E(k)$ as follows. 
Suppose that vertex $k$ (the next appearing vertex in $\G$) links to vertices~$i$ and~$j$ via edges $(i,k)$ and $(j,k)$. Without loss of generality, we assume that $(i,j)\in E_1$. There are two cases to consider:

Case I. Suppose that $ x_{i}$, $ x_j$, and $ x_{k}$ are aligned; then, we let 
$E'_1 := E_1 \cup \{(i,k), (j,k)\}$, and define the partition of $E(k)$ by 
$
E(k) =  E'_1 \cup E_2 \cup \ldots \cup E_m
$. 

Case II. Suppose that $ x_i$, $ x_j$, and $ x_{k}$ are not aligned; then, we define the partition of $E(k)$ by
 $
E(k) =  E_1 \cup \ldots \cup E_m \cup \{(i,k)\} \cup \{(j,k)\}
$. 

\end{enumerate}
Following the Henneberg sequence $\G$, an independent partition for  $(G,p)$ is derived. 
\end{Definition}

We illustrate the notion of  an independent partition in Fig.~\ref{DECO}. We show below that the independent partition is unique and, in particular, independent of the particular choice of a Henneberg sequence.

\begin{figure}[h]
\begin{center}
\includegraphics[width=0.6\textwidth]{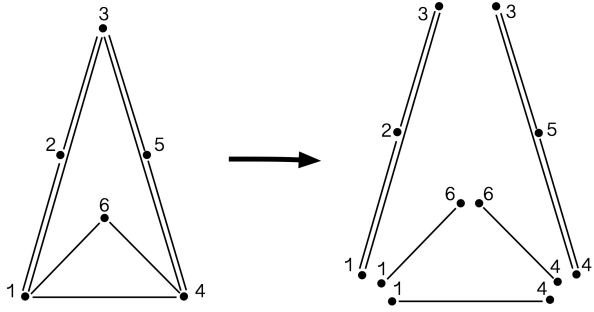}
\caption {Example of an independent partition. We see from the left figure that the graph $G$ is a triangulated Laman graph as we label the vertices with respect to a Henneberg sequence, and $p$ is a planar configuration with $ x_1,  x_2, x_3$ aligned, and $ x_3, x_4, x_5$ aligned. Then, the independent partition for $(G,p)$ is given by the right figure.}
\label{DECO}
\end{center}
\end{figure}


\begin{lem}\label{INDHC} 
Let $(G,p)$ be a framework, with $G$ a triangulated Laman graph and $p\in P_G$. Then, the independent partition for $(G,p)$ does not depend on any particular choice of a Henneberg sequence of~$G$. 
\end{lem} 

\begin{proof} The proof is  by induction on the number of vertices of $G$. For the base case $n = 2$, there is only one Henneberg sequence of~$G$, and hence Lemma~\ref{INDHC} is trivially true.  

For the inductive step, assuming that Lemma~\ref{INDHC} holds for $n <k$, with $k \ge 3$, we show its validity for $n = k$.
 Let $\G=\{G(l)\}^{k}_{l = 1}$ and $\G'=\{G'(l)\}^{k}_{l = 1}$ be two distinct Henneberg sequences of $G$. We want to show that the independent partitions, obtained by following $\G$ or $\G'$ (introduced in Definition~\ref{def:indpar}), are the same. Label the vertices of $G$ with respect to the order in which they appear in  $\G$, and assume that the last vertex to appear, vertex~$k$, is adjacent to vertices~$i$ and~$j$.  Hence, there is a $3$-cycle $(i,j,k)$ in $G$ and vertex~$k$ is not in any other $3$-cycle.

First, assume that either $(i,k)$ (or $(j,k)$) is the first edge appearing in $\G'$. Then, vertex~$j$ (or $i$) \emph{has} to be the next vertex appearing in $\G'$ so that (i) the $3$-cycle $(i,j,k)$ is formed and (ii) vertex~$k$ is not in any other $3$-cycle. This sequence $\G'$ yields the same independent partition as the sequence $\G''$ with $(i,j)$ as its first edge, vertex~$k$ the next vertex added, and  all other steps the same. A similar reasoning applies for the case where $(j,k)$ is the first edge of $\G'$. Thus,  we can assume, without loss of generality, that neither $(i,k)$ nor $(j,k)$ is the first edge of $\G'$.

Second, note that vertex~$k$ has degree $2$, and hence it must appear in $\G'$ by linking to vertices~$i$ and~$j$ via edges $(i,k)$ and $(j,k)$. Moreover, once vertex~$k$ joins $G$,  there will be no vertex linking to~$k$. Hence, it does not matter (for the independent partition) when vertex~$k$ appears in $\G'$, and we assume that $k$ is the last vertex joining $G$ in~$\G'$.    

Now, let  $G^*= (V^*, E^*)$ be the subgraph induced by vertices $V^* := \{1,\cdots,k-1\}$, and $\(G^*,p^*\)$ be the corresponding framework. 
Since $k$ is the last vertex joining $G$ in the sequences $\G$ and $\G'$,  we obtain Henneberg sequences $\G^*$ and $\G'^*$  of $G^*$ by omitting the last step of $\G$ and of $\G'$.   
From the induction hypothesis,  the two sequences $\G^*$ and $\G'^*$ give rise to the same independent partition of $E^*$ for $\(G^*,p^*\)$, which completes the proof.
\qquad\end{proof}

We describe below some  properties of the independent partition that will be needed to prove the main results. 
We first have some definitions and notations. 
Let $(G, p)$ be a framework, with $G$ a triangulated Laman graph and $p\in P_G$. Let $\{E_1,\cdots, E_m\}$ be the disjoint subsets of edges associated with the independent partition for $(G, p)$. Let $G_i =(V_i,E_i)$ be a subgraph of $G$, with $V_i$ the set of vertices incident to edges in $E_i$. Let $p_i$ for be a sub-configuration of $p$ induced by $G_i$.  We call $\left \{\(G_i, p_i\) \right\}^m_{i=1}$ {\it the frameworks associated with the independent partition for $(G, p)$}.  

For each $i = 1,\ldots, m$, we let $V_i = \{i_1,\ldots, i_{|V_i|}\}$. Recall that the configuration space $P_{G_i}$ associated with $G_i$ is defined as follows:
$$
P_{G_i} := \left\{  (x_{i_1}, \ldots, x_{i_{|V_i|}}) \in \R^{2|V_i|} \mid x_{i_j} \neq x_{i_k}  \right\}.
$$
We also recall that the group action of $\SE(2)$ on $P_G$ is defined in~\eqref{def:groupaction}.  In the same way, we let $\SE(2)$ act on $P_{G_i}$, and denote by $O_{p_i}$ the orbit of $p_i$ in $P_{G_i}$. The dimension of $O_{p_i}$ is~$3$. We further note that if $p'$ is in $O_p$ and $p'_i$ is the sub-configuration of $p'$ induced by $G_i$, then $p'_i $ is in $O_{p_i}$. 

Further, recall that for a vector $p = (x_1,\ldots, x _n)\in \R^{2n}$, with $x_i\in \R^2$, we have that $p|_{i} := x_i$ for all $i =1,\ldots,  n$. Similarly, for a subset $V' = \{i_1,\ldots, i_k\}$ of $V$, we define 
\begin{equation}\label{eq:defpsubv}
p|_{V'} := \(x_{i_1},\ldots, x_{i_k} \).
\end{equation}  
So, for example, if $\{(G_i,p_i\}^m_{i = 1}$, with $G_i = (V_i, E_i)$, are the frameworks associated with the independent partition for $(G,p)$, then $p |_{V_i} = p_i$ for all $i = 1,\ldots, m$. But the definition~\eqref{eq:defpsubv} is not restricted only to configurations in $P_G$, but rather applied to arbitrary vectors in $\R^{2n}$.

We call a frame $(G, p)$, with $p = (x_1,\ldots, x _n)$, a {\bf line framework} if all $x_i$, for $1 \leq i \leq  n$, are aligned, and $p$ a {line configuration}.  
With the definitions and notations above, we have the following result:

\begin{theorem}\label{indpar}
Let $(G,p)$ be a framework, with $G$ a triangulated Laman graph and $p \in P_G$.
Let $\{(G_i,p_i)\}^m_{i=1}$, with $G_i = (V_i, E_i)$,  be the frameworks associated with the independent partition for $(G,p)$. Then, the following properties hold:
\begin{enumerate}
\item\label{cond1}  Each $G_i$ is a triangulated Laman graph, and each $(G_i, p_i)$ is a line framework.

\item If there is another partition of $E$ satisfying  condition~\ref{cond1}, then it is a refinement of the independent partition.  In other words, the independent partition contains minimal number of sub-frameworks $(G_i,p_i)$  satisfying  condition~\ref{cond1}.


\item  For each $i = 1,\ldots, m$, there is an open neighborhood $\cal{U}_i$ of $p_i$  in $\mathbb{R}^{2|V_i|}$, and a unique smooth map 
\begin{equation}\label{DELTAMAP}
\eta_i : \cal{U}_i \longrightarrow P_G
\end{equation}
such that $\eta_i(p_i) = p$, and moreover, for  all $p'_i$ in $\cal{U}_i$, we have  
$$
 \eta_i(p'_i)|_{V_j} =
\left\{
 \begin{array}{ll}
   p'_i  & \mbox{if } j = i, \vspace{3pt}\\
   \in O_{p_j} & \mbox{otherwise}.
 \end{array}
 \right. 
 $$



\end{enumerate}\, 
\end{theorem}

The last part of Theorem~\ref{indpar} will be of great use to compute the signature of the Hessian of~$\Phi$ (in the next section). Roughly speaking, it says  that we can freely perturb the positions of agents in a sub-framework $(G_i,p_i)$ while keeping unchanged the inter-agent distances in each of the other frameworks $(G_j,p_j)$,  for $j \neq i$. For example, let $G$ be a complete graph of $3$ vertices, and  $p=(x_1,x_2,x_3)$ a non-degenerate triangle, i.e., $x_1$, $x_2$, and $x_3$ are not aligned. Then, the independent partition for $(G,p)$ yields 
$$\left\{E_1:=\{2,3\},E_2:=\{1,3\},E_3:=\{(1,2)\}\right\}.$$ 
It is intuitively clear in this case that for any small perturbation $p'_3=(x'_1,x'_2)$ of $p_3=(x_1,x_2)$, there exists a unique $x_3'$, close enough to~$x_3$,  such that  
$\|x'_1-x_3'\|=\|x_1-x_3\|$ and   $\|x_2'-x_3'\|=\|x_2-x_3\|$.   
The map $\eta_3$ then yields $x_3'$ as a function of $x'_1$ and $x_2'$, i.e., $\eta_3(p_3')=(x_1',x_2',x_3')$. A precise statement for this fact will be in Lemma~\ref{SIMPLE}, Subsection~\S\ref{Shape}.

With the last part of Theorem~\ref{indpar}, we can easily evaluate the differential of the maps $\eta_i$. 
First, for an arbitrary smooth manifold $M$, denote by $T_pM$ be the {\it tangent space} of $M$ at a point $p\in M$.  Since $\eta_i$ is smooth, the derivative of $\eta_i$ at $p_i$ is well defined. We denote it by 
$$ D_{p_i}\eta_i : T_{p_i} \cal{U}_i \longrightarrow T_{p}P_G,$$ 
 where we have used the property that $\eta_i(p_i) = p$. 
The following fact is then an immediate consequence of the last part of Theorem~\ref{indpar}: 

\begin{cor}\label{cor:derivativeofeta}
Let $(G,p)$ be a framework, with $G$ a triangulated Laman graph and $p \in P_G$.  Let $\left \{ (G_i, p_i) \right \}^m_{i=1}$, with $G_i = (V_i, E_i)$, be frameworks associated with the independent partition for $(G,p)$. Let $\eta_i$ be defined in Theorem~\ref{indpar}, and $D_{p_i}\eta_i$ be the derivative map. Then,  for any $v\in T_{p_i} \cal{U}_i$, we have
$$
D_{p_i}\eta_i(v) |_{V_j} = 
\left \{
\begin{array}{ll} 
v & \mbox{if } j = i,  \vspace{3pt}\\
\in T_{p_j}O_{p_j} &  \mbox{otherwise}.
\end{array}\right. 
$$\,
\end{cor}

In the remainder of this section, we establish properties of independent partitions that are needed to prove Theorem~\ref{indpar}. The first part of Theorem~\ref{indpar} follows from the definition of independent partitions. We prove parts~2 and~3 of Theorem~\ref{indpar} in Subsections~\S\ref{UNIQS} and~\S\ref{Shape}, respectively.

\subsection{On edge-set partitions}\label{UNIQS}
In this subsection, we establish the second part of Theorem~\ref{indpar}. 
Let $G = (V,E)$ be a triangulated Laman graph, and $(G,p)$ be a framework.  Let $\Sigma(G,p)$ be the collection of partitions  of the edge-set $E$ of $G$ that satisfy the condition in part~1 of Theorem~\ref{indpar}. 

There is a natural partial order on $\Sigma(G,p)$ reflecting the granularity of the partition. Specifically, let  $\sigma = \{E_1,\cdots, E_m\}$ and $\sigma' = \left \{E'_1,\cdots, E'_{m'} \right \}$ be two different partitions in $\Sigma(G,p)$; then, we say that $\sigma$ is {\bf coarser} than $\sigma'$, or simply write $\sigma\succ\sigma'$,  if each $E'_i$ is a subset of some $E_j$. 
An element $\sigma$ in $\Sigma(G,p)$ is said to be a {\bf maximal} (resp. {\bf minimal}) element if there does not exist an element $\sigma' \neq \sigma$ in $\Sigma(G,p)$ such that $\sigma' \succ \sigma$ (resp. $\sigma' \prec \sigma$).  
Note that $\Sigma(G,p)$ has a {\it unique} minimal element $\sigma_- := \{(i,j) \mid (i,j)\in E \}$, i.e., each subset $E_k$ of the partition is a singleton $\{(i,j)\}$, for $(i,j)\in E$. The independent partition is by  definition an element of $\Sigma(G,p)$. We now show that it is in fact the unique maximal  element according to the partial order `$\succ$':

\begin{pro}\label{PARMAX} 
Let $(G,p)$ be a framework, with $G$ a triangulated Laman graph and $p\in P_G$, and let $\Sigma(G,p)$ be the (partially ordered) set of partitions of $E$ satisfying the condition in part~1 of Theorem~\ref{indpar}.  Then, the independent partition for $(G,p)$ is the unique maximal element in $\Sigma(G,p)$.      
\end{pro}

We establish below Proposition~\ref{PARMAX}.  Let $G = (V,E)$ be a triangulated Laman graph of $n $ vertices, and let $\G = \{G(l)\}^n_{l = 1}$ be a Henneberg sequence of $G$. Let 
$G'  = (V',E')$ be a subgraph of $G$; we say that $G'$ is a {\bf leading subgraph} of $\G$ if $G'=G(|V'|)$. 
We have the following fact:

\begin{lem}\label{BASEIR} 
Let $G = (V, E)$ and $G' = (V',E')$ be triangulated Laman graphs of at least two vertices, with $G'$ a subgraph of $G$. Then,  there is a Henneberg sequence of $G$ with $G'$ a leading subgraph.   
\end{lem}

\begin{proof} 
The proof is carried out by induction on the number of vertices of $G$. For the base case $n = 2$, the lemma is trivially true. For the inductive step, assuming that the lemma holds for $n <k$, we prove it for $n =k$. 
Without loss of generality, we assume that vertex $k$ is of degree~$2$, adjacent to vertices~$i$ and~$j$.  
Let $G^* = (V^*, E^*)$ be the subgraph of $G$ induced by vertices $V^* := \{1,\cdots, k-1\}$. There are two cases to consider: 

{\it Case I}. Suppose that $k$ is  not a vertex of $G'$; then $G'$ is a subgraph of $G^*$. By the induction hypothesis,  we can choose a Henneberg sequence of $G^*$ with $G'$ its leading subgraph.  Then, following the sequence, we construct $G$ by joining the vertex $k$ to vertices~$i$ and~$j$.  

{\it Case II}. Suppose that $k$ is a vertex of $G'$. Let $k'$ be the number of vertices of $G'$. We first assume that $k' \ge 3$. Note that the degree of vertex~$k$ is $2$, and $k$ is adjacent to vertices~$i$ and $j$.  Hence, the graph $G'$ contains all the three vertices~$i$, $j$, and $k$. Let $G'' = (V'', E'')$ be a subgraph of $G'$ induced by $V'': = V' - \{k\}$. Then, $G''$ is a graph of $(k'-1)$ vertices; in particular, it contains vertices~$i$ and~$j$. Since $G''$ is a subgraph of $G^*$, by the induction hypothesis, there is a Henneberg sequence 
$
\G^* = \{G^*(l)\}^{k-1}_{l = 1}
$ 
of $G^*$, 
with $G''$ a leading subgraph, i.e., $G^*(k'-1) = G''$. We further note that for all $l \ge k'$, the graph $G^*(l)$ contains the two vertices~$i$ and~$j$. 
We now define a Henneberg sequence $\G = \{G(l)\}^k_{l = 1}$ of $G$ as follows: 
\begin{enumerate}
\item For $l = 1,\ldots, k'-1$, let  $G(l) := G^*(l)$. 
\item For $l = k'$, let $G(k') := G'$; 
\item For $l = k' + 1, \ldots, k$, let $G(l)$ be defined by attaching vertex~$k$ to vertices~$i$ and~$j$ in $G^*(l - 1)$ via edges $(i,k)$ and $(j,k)$.
\end{enumerate}
From the construction, we conclude that $\G$ is a Henneberg sequence of $G$ with $G'$ a leading subgraph. 

We now assume that $k' = 2$, i.e., $G'$ has only two vertices. Without loss of generality, we assume that $G'$ is comprised of vertices~$i$ and~$k$, with a single edge $(i,k)$. Let $C$ be the $3$-cycle $(i,j,k)$; then clearly $G'$ is a subgraph of $C$. By the arguments above, there exists a Henneberg sequence of $\G = \{G(l)\}^k_{l = 1}$ with $C$ a leading subgraph, i.e., $G(3) = C$. We then modify $G(1)$ and $G(2)$ in the sequence, if necessary, such that $G(1)$ is the graph comprised of a single vertex~$i$, and $G(2) = G'$.  This completes the proof. 
\qquad\end{proof}

With Lemma~\ref{BASEIR} at hand, we prove Proposition~\ref{PARMAX}.

{\em Proof of Proposition~\ref{PARMAX}}. Let $\sigma =\{E_1,\cdots,E_m\}$ be the independent partition for $(G,p)$. Given another arbitrarily chosen partition $\sigma'=\{ E'_1,\cdots,E'_{m'}\}$ in $\Sigma(G,p)$, we show that $\sigma \succ\sigma'$. Denote by $\{G_i = (V_i,E_i), p_i\}^m_{i=1}$ (resp. $\{G'_i = (V'_i,E'_i), p'_i\}^{m'}_{i = 1}$) the frameworks  associated with $\sigma$ (resp. $\sigma'$). It suffices to  show that each $G'_i$ is a subgraph of $G_j$ for some $j$.  Since $G'_i$ is, by assumption, a triangulated Laman graph,  we conclude from Lemma~\ref{BASEIR} that there is a Henneberg sequence $\G$ of $G$ with $G'_i$ a leading subgraph. Now, suppose that $G'_i$ has $k$ vertices, with $V'_i = \{i_1,\cdots, i_k\}$, and that $\(i_1,i_2\)$ is the first edge appearing in $\G$. Let $G_j$ be the subgraph which contains
$\(i_1,i_2\)$ as an edge.  Note that $(G'_i,p'_i)$ is, by assumption, a line framework, and hence, by using $\G$ to construct the independent partition as introduced in Definition~\ref{def:indpar}, we conclude that all the edges of $G'_i$ are edges of $G_j$. This completes the proof.  
\qquad\endproof

\begin{Remark}\label{rmk:indepedentpartitionforstronglyconf}\normalfont
When  $p$ is strongly rigid,  the independent partition for $(G,p)$ is simply given by $\{(i,j) \mid (i,j)\in E \}$. Then, the maximal element 
of $\Sigma(G,p)$ coincides with the minimal element, and hence the set $\Sigma(G,p)$ is a singleton. $\qed$   
\end{Remark}

The second part of Theorem~\ref{indpar}  follows from Proposition~\ref{PARMAX}. 

\subsection{Partition-adapted perturbations}\label{Shape} 

In this subsection, we show that we can perturb the positions of the agents in a sub-framework while keeping the shapes of the other sub-frameworks unchanged. We refer to such perturbations as {\it perturbations adapted to an independent partition}. This will establish the last part of Theorem~\ref{indpar}. We first consider the simple case when $p $ is comprised of only three agents. Let a configuration $p = (x_1, x_2, x_3)\in \R^6$, with $x_i \in \R^2$ for $i = 1,2, 3$,  be a {\it nondegenerate} triangle. The following result formalizes the fact that  the positions of $x_1,x_2$, and $x_3$ can be perturbed while keeping the distances $\|x_3-x_1\|$ and $\|x_3-x_2\|$ fixed:

\begin{lem}\label{SIMPLE} 
Let $ x_1$, $ x_2$, and $ x_3$ be in $\R^2$, which form a nondegenerate triangle. Then, there is an open neighborhood $U$ of $(x_1,x_2)$  in $\R^4$, and a unique smooth map 
$\xi: U \longrightarrow \R^2$ 
such that $\xi(x_1, x_2) = x_3$, and moreover, for any $(x'_1, x'_2)$ in $U$, we have 
$$
\left\{
\begin{array}{l}
\left \|\xi(x'_1,x'_2) - x'_1 \right \|  = \| x_3 - x_1 \|, \vspace{3pt}\\
\left \|\xi(x'_1,x'_2) - x'_2 \right \|  = \| x_3 - x_2 \|.
\end{array}
\right.
$$\,
\end{lem}

\begin{proof} 
Choose open neighborhoods  $U$ of $(x_1,x_2)$ in $\mathbb{R}^4$ and  $V$ of $x_3$ in $\R^2$ such that any triangle $p'=( x'_1, x'_2, x'_3)$, for $(x'_1,x'_2)\in U$ and $x'_3 \in V$,   is nondegenerate. Define a smooth function $h:U\times V\longrightarrow \mathbb{R}^2$ by 
\begin{equation*}
h( x'_1, x'_2, x'_3) := \frac{1}{2}
\begin{pmatrix}
\| x'_3- x'_1\|^2-\| x_3- x_1\|^2 \\
\| x'_3- x'_2\|^2 - \| x_3- x_2\|^2
\end{pmatrix}.
\end{equation*}
By computation, the derivative of $h$ with respect to $x'_3$ is given by 
\begin{equation}\label{DPHI}
\frac{\partial h( x'_1, x'_2, x'_3)}{\partial  x'_3}=
\begin{pmatrix}
 x'^\top_3- x'^\top_1\\
 x'^\top_3- x'^\top_2
\end{pmatrix} \in \R^{2\times 2}, 
\end{equation}
which is nonsingular for all $p' = (x'_1,x'_2,x'_3)$ in $U\times V$.  
Then, by the inverse function theorem,  there exists a unique smooth function $\xi: U \to V$ such that $\xi( x_1, x_2) =  x_3$, and moreover,  
\begin{equation*}
h\( x'_1,  x'_2,\xi( x'_1, x'_2)\) = 0.
\end{equation*}
This completes the proof.
\qquad\end{proof}

To establish the last part of Theorem~\ref{indpar}, we further need the following fact. First, recall that the distance function $\rho_{G}: P_{G}\longrightarrow \mathbb{R}^{|E|}_+$ is given by  
\begin{equation*}
\rho_{G}(p) := \(\cdots, \| x_i -  x_j\|^2,\cdots\)_{(i,j)\in E}.  
\end{equation*} 
It should be clear that if $p'\in O_p$, then $\rho_G(p') = \rho_G(p)$.  In other words, the orbit $O_p$ is a subset of the pre-image $\rho^{-1}_G(\rho_G(p))$. In general, the pre-image $\rho^{-1}_G(\rho_G(p))$ contains more than one orbit since there are more than one non-congruent framework with the same edge lengths. When $(G,p)$ is a line framework, however, matters simplify greatly as shown in the next lemma:

\begin{lem}\label{lem:preimagecontainsasingleorbit}
Let $(G, p)$ be a line framework, with $G$ a triangulated Laman graph and $p\in P_G$. Then, 
$
\rho^{-1}_G(\rho_G( p ) ) = O_p.
$ 
\end{lem}  

\begin{proof}
Let $p = (x_1,\ldots, x _n)$; since $p$ is a line configuration, its orbit $O_p$ can be characterized as follows: a configuration $p' = (x'_1,\ldots, x' _n)$ is in $O_{p}$ if and only if  there exist a vector $v$ in $\R^2$, with $\|v\| = 1$, such that
\begin{equation}\label{eq:lineconfigurationorbit}
x'_{i} - x'_1 =   \|x_{i} - x_{1}  \| \, v, \hspace{10pt} \forall\, i = 2,\ldots, n.
\end{equation}
We now show that if a configuration $p' $ lies in $\rho^{-1}_G(\rho_G(p))$, then there is a vector $v$ of unit length such that $p'$ satisfies~\eqref{eq:lineconfigurationorbit}. 

We prove this fact by induction on the number of vertices of $G$. For the base case $n = 2$, we have $p = (x_1, x_2)$, with $x_1 \neq x_2$. If $p' = (x'_1,x'_2)$ lies in $\rho^{-1}_G(\rho_G(p))$, then
$\|x'_2 - x'_1\| = \|x_2 - x_1\|$, and hence  
$$
x'_2 - x'_1 = \|x_2 - x_1\| \frac{x'_2 - x'_1}{\|x'_2 - x'_1\|}, 
$$
which implies that $p'\in O_p$. For the inductive step, assuming that the statement holds for $n < k-1$, with $k \ge 3$, we prove it for $n = k$. Without loss of generality, we assume that the vertex $k$ is of degree $2$, adjacent to vertices~$i$ and~$j$. Let $G^* = (V^*, E^*)$ be the subgraph induced by $V^* := \{1,\ldots, k-1\}$, and $p^*\in P_{G^*}$ be the sub-configuration of $p$. Let $p' $ be in $\rho^{-1}_{G}(\rho_G(p))$, and $p'^* \in P_{G^*}$ be the sub-configuration of $p'$. Note that $(G^*,p^*)$ is a line framework, with $G^*$ a triangulated Laman graph, and moreover, $\rho_{G^*}(p'^*) = \rho_{G^*}(p^*)$. Hence, by the induction hypothesis, we have a vector~$v$ of unit length such that   
\begin{equation}\label{eq:inductiongetdistanceequal}
x'_i - x'_1 = \|x_i - x_1\|\, v, \hspace{10pt} \forall i = 2,\ldots, k-1. 
\end{equation}
On the other hand, for the $3$-cycle $(i,j,k)$, we have 
\begin{equation}\label{eq:degeneratetriangle}
\left\{
\begin{array}{l}
\|x'_k - x'_i \| = \|x_k - x_i\|, \\
\|x'_k - x'_j\| = \|x_k - x_j\|, \\
\|x'_i - x'_j\| = \|x_i - x_j\|.
\end{array}
\right.
\end{equation}
This, in particular, implies that the two triangles formed by $(x_i,x_j,x_k)$ and by $(x'_i,x'_j,x'_k)$ are congruent.  Since $(x_i,x_j,x_k)$ is aligned, so is $(x'_i,x'_j,x'_k)$. This, in particular, implies that there is a unique solution for $x'_k$ such that~\eqref{eq:degeneratetriangle} holds, which is given by 
$
x'_k = x'_1 + \|x_k - x_1\| \, v 
$. 
Combining this fact with~\eqref{eq:inductiongetdistanceequal}, we then conclude that $p'\in O_p$. 
\qquad\end{proof}

With Lemmas~\ref{SIMPLE} and~\ref{lem:preimagecontainsasingleorbit} at hand, we now prove the last part of Theorem~\ref{indpar}:

{\em Proof of the last part of Theorem~\ref{indpar}}. 
First, note that from Lemma~\ref{lem:preimagecontainsasingleorbit}, it suffices to prove that for each $i = 1,\ldots, m$, there exists an open neighborhood $\cal{U}_i$ of $p_i$ in $P_{G_i}$, and a unique smooth map $\eta_i: \cal{U}_i\longrightarrow P_G$ such that the following hold: 
\begin{enumerate}
\item $\eta_i(p_i) = p$.
\item For all $p'_i \in \cal{U}_i$, we have 
\begin{equation}\label{eq:keepidentityonitself}
\eta_i(p'_i) |_{V_i} = p'_i,
\end{equation} 
and for all $(j,k)\notin E_i$, 
\begin{equation}\label{eq:keepingdistances}
\|\eta_i(p'_i) |_j - \eta_i(p'_i) |_k \|= \|x_j - x_k\|.
\end{equation}
\end{enumerate}\,
To see this, note that if~\eqref{eq:keepingdistances} holds, then for any $j \neq i$, we have 
$
\eta_i(p'_i) |_{V_j} \in \rho^{-1}_{G_j}(\rho_{G_j}(p_j))
$, 
and hence by Lemma~\ref{lem:preimagecontainsasingleorbit}, we have $\eta_i(p'_i) |_{V_j} \in O_{p_j}$. 

We now prove the statement above. In the remainder of the proof, we fix $i = 1$.  From Lemma~\ref{BASEIR}, there is a Henneberg sequence  of $G$:
$$\G = \{G(l)=  (V(l), E(l)) \}^n_{l = 1}$$  such that  $G_1$ is a leading subgraph of $\G$. The map $\eta_1$ is then defined along $\G$.  
Label the vertices of $G$ with respect to the order they appear in the sequence.  
Suppose that $G_1 = (V_1, E_1)$ has $n_1$ vertices, with $V_1 = \{1,\ldots, n_1\}$. We choose an open neighborhood $\cal{U}_1$ of $p_1$  in $\mathbb{R}^{2n}$, and along the proof, we may shrink $\cal{U}_1$ if necessary. Let $p'_1 = (x'_1,\ldots, x'_{n})$ be in $\cal{U}_1$. We define $\eta_1(p'_1)$ by subsequently specifying $\eta_1(p'_1)| _i$ for $i =1,\ldots,  n$.

{\it Initial step}. Starting with $G(n_1) = G_1$, we define $\eta_1(p'_1)|_i = x'_i$, for $i = 1,\ldots, n_1$, so that~\eqref{eq:keepidentityonitself} is satisfied. 

{\it Recursive step}. We assume that  $\eta_{1}(p'_1) |_i$, for $i\le k-1$ (with $k > n_1$),  are smoothly defined such that $\eta_1(p_1) |_i = x_i$, and moreover, 
\begin{equation*}
\|\eta_1(p'_1) |_i - \eta_1(p'_1)|_j\| = \|x_i - x_j\| 
\end{equation*}
for all $(i, j) \in  E(k-1) - E(n_1)$. 
Suppose that in the sequence $\G$, the vertex~$k$ links to vertices~$i$ and~$j$. 
We now show that $\eta_1(p'_1)|_k$ can be smoothly defined such that 
\begin{equation}\label{eq:checkingcondition1}
\eta_1(p_1)|_k = x_k,
\end{equation}
and moreover, 
\begin{equation}\label{eq:checkingcondition2}
\left\{
\begin{array}{l} 
\| \eta_1(p'_1)|_{k} - \eta_1(p'_1) |_{i} \| = \| x_k - x_i \|, \vspace{3pt}\\
\| \eta_1(p'_1)|_{k} - \eta_1(p'_1) |_{j} \| = \| x_k - x_j \|.
\end{array}
\right.
\end{equation}
There are two cases to consider:

\textit{Case I}. Suppose that $ x_i$, $ x_j$, and $ x_k$ are not aligned in $p$; then, the three agents form a nondegenerate triangle. From Lemma~\ref{SIMPLE}, there are open neighborhoods $U_i$, $U_j$, and $U_k$ of $x_i$, $x_j$, and $x_k$ in $\R^2$, and a unique smooth map  
$
\xi: U_i \times U_j \longrightarrow U_k
$ 
such that $\xi(x_i, x_j) = x_k$, and moreover, for any $(x'_i, x'_j)$ in $U_i\times U_j$, we have 
\begin{equation*}\label{eq:xiij}
\left\{
\begin{array}{l}
\|\xi(x'_i,x'_j) - x'_i \|  = \| x_k - x_i \|, \vspace{3pt}\\
\|\xi(x'_i,x'_j) - x'_j \|  = \| x_k - x_j \|.
\end{array}
\right.
\end{equation*}
From the induction hypothesis, both $\eta_{1}(p'_1) |_i$ and $\eta_{1}(p'_1) |_j$ are well defined, with 
$\eta_1(p_1) |_{i} = x_i$  and  $\eta_1(p_1) |_{j} = x_j$.  
Moreover, by shrinking $\cal{U}_1$ if necessary, we have that $\eta_{1}(p'_1) |_i \in U_i$ and  $\eta_{1}(p'_1) |_j \in U_j$ for all $p'_1 \in \cal{U}_1$. 
We can thus define $$\eta_1(p'_1)|_{k} := \xi\(\eta_{1}(p'_1) |_i,\,  \eta_{1}(p'_1) |_j\).$$
By combining Lemma~\ref{SIMPLE} with the induction hypothesis, we have that $\eta_1$ is smooth in~$p'_1\in \cal{U}_1$, and satisfies~\eqref{eq:checkingcondition1} and~\eqref{eq:checkingcondition2}. 


\textit{Case II}. Suppose that $ x_i$, $ x_j$, and $ x_k$ are aligned; we define $\eta_1(p'_1)|_k$ as follows: first, let $c_i$ and $c_j$ be two scalars defined as follows:
$$
c_i := \frac{\langle x_j - x_k, x_j - x_i \rangle}{\|x_j - x_i \|^2} \hspace{5pt} \mbox{ and } \hspace{5pt} c_j := \frac{\langle x_i - x_k, x_i - x_j \rangle}{\|x_i - x_j \|^2},
$$
where $\langle \cdot, \cdot \rangle$ is the standard inner-product in $\R^n$; 
we then set 
\begin{equation}\label{ALIGND}
\eta_1(p'_1)|_k:= c_i\, \eta_1(p'_1)|_i + c_j \, \eta_1(p'_1)|_j. 
\end{equation}  
From the induction hypothesis, we have that $\eta_1(p'_1) |_k$ is smooth in $p'_1\in \cal{U}_1$, and moreover, satisfies~\eqref{eq:checkingcondition1} by the choices of~$c_i$ and~$c_j$. Also, note that $\eta_1(p'_1)|_i $, $\eta_1(p'_1)|_j$, and $\eta_1(p'_1)|_k$ are aligned; indeed, by computation,  we have
\begin{equation}\label{eq:alignedetaijk}
\left\{
\begin{array}{l}
\eta_1(p'_1)|_k - \eta_1(p'_1)|_i = c_j \(\eta_1(p'_1)|_j - \eta_1(p'_1)|_i \),\vspace{3pt}\\
\eta_1(p'_1)|_k - \eta_1(p'_1)|_j = c_i \(\eta_1(p'_1)|_i - \eta_1(p'_1)|_j \).
\end{array}
\right. 
\end{equation}
Appealing again to the induction hypothesis, we have
$$
 \| \eta_1(p'_1)|_j - \eta_1(p'_1)|_i \| = \|x_j - x_i \|. 
$$ 
Combining this fact with~\eqref{eq:alignedetaijk}, we conclude that~\eqref{eq:checkingcondition2} holds.  
\qquad\endproof

\section{Evaluating the signatures of the Hessians}\label{ssec:morsebott}

In this section, we use independent partitions to evaluate the Hessians of the potential function $\Phi$ at the critical points. We also prove Theorem~\ref{MAIN}.  

\subsection{Independent partitions for critical points}\label{OnEC}
Recall that the  dynamics of agents $x_i$, for $i\in V$,  are given by
\begin{equation}\label{eq:relaxeddynamics}
\dot{x}_i = \sum_{j \in \cal{N}_i} u_{ij}(d_{ij},\ol d_{ij}) (x_j - x_i),  \hspace{10pt} \forall i\in V, 
\end{equation}
and the equilibria of this system are the critical points of the potential $\Phi$ defined by~\eqref{PHI}.
We establish the following fact:

\begin{pro}\label{pro:equilibriumcondition}
Let $G$ be a triangulated Laman graph, and $p\in P_G$ be an equilibrium of system~\eqref{eq:relaxeddynamics}. Let $\{(G_i, p_i)\}^m_{i=1}$, with $G_i = (V_i, E_i)$,  be the frameworks associated with the independent partition for $(G, p)$. Then, each $p_i$ is an equilibrium of the subsystem of~\eqref{eq:relaxeddynamics} induced by $G_i$.
\end{pro}     

For Proposition~\ref{pro:equilibriumcondition}, we can actually relax the condition that each control law is a monotone att./rep. function, but require  only that  $u_{ij}(\cdot, \ol d_{ij})$ be a continuously differentiable function. We establish below Proposition~\ref{pro:equilibriumcondition} in this general context. 

\paragraph{System reduction} We introduce below an operation on system~\eqref{eq:relaxeddynamics} which helps to derive a new system of fewer agents. 
We have the following fact: when several agents are aligned, we can remove some of these agents from the formation and modify the interactions between the remaining agents. We call this construction, which we make precise below, \emph{system reduction}. We start with  some definitions and notation. First, for ease of notation, we let 
$$f_{ij}(\cdot):= u_{ij}(\cdot, \ol d_{ij}), \hspace{10pt} \forall (i,j)\in E.$$ 
Let $(G, p)$ be a framework, with $G$ a triangulated Laman graph. 
Choose a vertex~$k$ of $G$ of degree~$2$, and assume that vertex~$k$ is adjacent to vertices~$i$ and~$j$.  
Now, suppose that  $x_i$, $x_j$, and $x_k$ are aligned, and moreover, that the dynamics of $x_k$ (in system~\eqref{eq:relaxeddynamics}) is zero at~$p$, i.e., 
\begin{equation}\label{eq: defeqg1}
\dot x_k = f_{ik}(d_{ik})(x_i - x_k) + f_{jk}(d_{jk})(x_j - x_k) = 0.  
\end{equation}
Introduce a function $g_{ij}\in \operatorname{C}^1(\mathbb{R}_+,  \mathbb{R})$  such that the value of $g_{ij}$ at $d_{ij}$ satisfies
\begin{equation}\label{eq:defgproperty1}
g_{ij}(d_{ij})( x_j -  x_i)  =  f_{ik}(d_{ik}) ( x_k -  x_i),  
\end{equation} 
but is arbitrary otherwise. 
Note that such a function exists because the vectors $(x_j - x_i)$ and $(x_k - x_i)$ are nonzero and moreover linearly dependent. 
Also, note that 
from~\eqref{eq: defeqg1} and~\eqref{eq:defgproperty1}, we have 
\begin{equation}\label{eq:defgproperty2}
g_{ij}(d_{ij}) ( x_j -  x_i)= f_{jk}(d_{jk}) ( x_j -  x_k).
\end{equation}
Let $G^* = (V^*, E^*)$ be the subgraph induced by $V^* := V - \{k\}$, and $(G^*,p^*)$ be the corresponding framework.  
Let $\cal{R}$ be a formation control system induced by $G^*$, with the control laws denoted by $ f^*_{i'j'}$, for $(i',j') \in E^*$. We say that $\cal{R}$ is a {\bf reduction of~\eqref{eq:relaxeddynamics} for $(G,p)$} if  the control laws $f^*_{i'j'}$ are such that 
\begin{equation}\label{eq:defwidetildefij}
f^*_{i'j'} := 
\left\{
\begin{array}{ll}
f_{ij} + g_{ij} & \mbox{if } (i',j') = (i,j) \\
f_{i'j'} &  \mbox{if } (i',j') \in E^* - \{(i,j)\}, 
\end{array}
\right. 
\end{equation}
with $g_{ij}\in \operatorname{C}^1(\R_+,\R)$ defined above. 

The main property of the reduced system is the following:


\begin{lem}\label{lem:systemreduction}
Suppose that $\cal{R}$ is a reduction of system~\eqref{eq:relaxeddynamics} for $(G,p)$;  then, $p^*$ is an equilibrium of $\cal{R}$ \emph{if and only if} $p$ is an equilibrium of system~\eqref{eq:relaxeddynamics}.   
\end{lem}

\begin{proof}
It suffices to show that the dynamics of $x_i$ and $x_j$ in system~\eqref{eq:relaxeddynamics} at $p$ are the same as they are in~$\cal{R}$ at $p^*$. The dynamics of $x_i$ in~$\cal{R}$ at $p^*$ is given by
\begin{equation}\label{eq:dynamicsofxiinsstar}
\dot x_i = f^*_{ij}(d_{ij})(x_j - x_i) + \sum_{j'\in \cal{N}^*_i \backslash \{j\}} f_{ij'}(d_{ij'})(x_{j'} - x_i),
\end{equation}
where $\cal{N}^*_i$ is the set of neighbors of vertex~$i$ in $G^*$.  
Combining~\eqref{eq:defgproperty1} and~\eqref{eq:defwidetildefij}, we have 
$$
f^*_{ij}(d_{ij})(x_j - x_i)  = f_{ij}(d_{ij})(x_j - x_i) + f_{ik}(d_{ik})(x_k - x_i),
$$
and hence~\eqref{eq:dynamicsofxiinsstar} is reduced to
$$
\dot x_i = \sum_{j'\in \cal{N}_i} f_{ij'}(d_{ij'})(x_{j'} - x_i),
$$
which is exactly the dynamics of $x_i$ in system~\eqref{eq:relaxeddynamics} at $p$. 
For $x_j$, we apply the same arguments as above (using~\eqref{eq:defgproperty2} instead of~\eqref{eq:defgproperty1}), and conclude that the dynamics of $x_j$ in~\eqref{eq:relaxeddynamics} at $p$ is the same as it is in $\cal{R}$ at $p^*$. This completes the proof. 
\qquad\end{proof}

With Lemma~\ref{lem:systemreduction} at hand,  we now prove Proposition~\ref{pro:equilibriumcondition}:

{\em Proof of Proposition~\ref{pro:equilibriumcondition}}. 
The proof is  by induction on the number of vertices of the graph. For the base case $n = 2$, the proposition is trivially true.  For the inductive step, we assume that the proposition holds for $n <k$ and prove it for $n =k$. Choose a Henneberg sequence $\G$ of $G$, and label vertices of $G$ with respect to the order in which they appear in the sequence. We assume that vertex~$k$ links to vertices~$i$ and~$j$. 

Let $G^*$ be the subgraph induced by vertices $\{1,\cdots,k-1\}$, and $\(G^*, p^*\)$ be the corresponding framework. 
There are two cases to consider: 

\textit{Case I}. Suppose that $ x_i$, $ x_j$, and $ x_k$ are not aligned; then, from the definition of independent partition,  there are two singletons $\{(i,k)\}$ and $\{(j,k)\}$ in the subsets $\{E_l\}^m_{l=1}$. Without loss of generality, we assume that $E_1 = \{(i,k)\}$ and $E_2 = \{(j,k)\}$. Note that $\{E_l\}^m_{l=3}$ is then the independent partition for the sub-framework $(G^*,p^*)$.  Let $\cal{S}_l$ be the subsystem induced by $G_l$, for $l = 1,\ldots, m$; we show that each $p_l$ is an equilibrium of $S_l$. 
Since $p$ is an equilibrium of system~\eqref{eq:relaxeddynamics}, the dynamics of $x_k$ is zero at~$p$, i.e.,
\begin{equation}\label{eq:xkdotzero}
\dot{ x}_k=f_{ik}(d_{ik})( x_i- x_k) +f_{jk}(d_{jk})( x_j- x_k) = 0. 
\end{equation} 
Combining~\eqref{eq:xkdotzero} with the fact that the vectors $( x_i- x_k)$ and $( x_j- x_k)$ are linearly independent, we obtain 
\begin{equation}\label{eq:finfjn0}
f_{ik}(d_{ik}) = f_{jk}(d_{jk}) = 0,
\end{equation} 
which implies that $p_{1}$ and $p_{2}$ are equilibria of the subsystems induced by  $G_{1}$ and of $G_{2}$, respectively.  
Furthermore, following~\eqref{eq:finfjn0}, we have that $p^*$ is an equilibrium of the subsystem induced by $G^*$. Appealing to the induction hypothesis, we have that each $p_l$  is an equilibrium of $\cal{S}_l$ for all $l = 3,\ldots, m$. We have thus established the proposition for the first case.

\textit{Case II}. Suppose  that $ x_i$, $ x_j$, and $ x_k$ are aligned. Then,  there is a subset, say  $E_1$, containing $(i,j)$, $(i,k)$ and $(j,k)$. Let $E^*_l$, for $l = 1,\ldots, m$, be defined such that 
$$
E^*_l = 
\left\{
\begin{array}{ll}
E_1 - (i,k) - (j,k) & \mbox{if } l = 1,\\
E_l & \mbox{otherwise}. 
\end{array}
\right.
$$ 
Note that $E^*_1$ is nonempty because it contains $(i,j)$. We also note that $\{E^*_i\}^m_{i=1}$ is the independent partition for $(G^*,p^*)$. 
Let $\{(G^*_i,p^*_i)\}^m_{i=1}$ be the frameworks associated with that independent partition. It should be clear that $(G^*_1,p^*_1)$ is a sub-framework of $(G_1,p_1)$, and $(G^*_l,p^*_l) = (G_l,p_l)$ for all $l > 1$.   
Because $x_i$, $x_j$, and $x_k$ are aligned and 
the dynamics of $x_k$ is zero at $p$,   
we can define a formation control system $\cal{R}$ which is induced by $G^*$, and is a reduction of system~\eqref{eq:relaxeddynamics} for $(G, p)$ (as in~\eqref{eq:defwidetildefij}). Since $p$ is an equilibrium of system~\eqref{eq:relaxeddynamics}, from Lemma~\ref{lem:systemreduction}, $p^*$ is an equilibrium of $\cal{R}$.   Let $\cal{R}_l$ be the formation control subsystem of $\cal{R}$ induced by $G^*_l$, for  $l = 1,\ldots,m$.  From the induction hypothesis, each $p^*_l$ is an equilibrium of $\cal{R}_l$ for all $l = 1,\ldots,m$.  Note that for $l >1$, $\cal{S}_l = \cal{R}_l$, and hence $p_l$ is an equilibrium of $\cal{S}_l$; for $i = 1$, we have that $\cal{R}_{1}$ is a reduction of $\cal{S}_1$ for $(G_1, p_1)$. Appealing again to Lemma~\ref{lem:systemreduction}, we conclude that $p_1$ is an equilibrium of $\cal{S}_1$. This completes the proof. 
\qquad\endproof


\subsection{Independent partitions for evaluating the signatures  of Hessians}
We evaluate in this subsection the signatures of the Hessians of the potential function~$\Phi$. First, recall that the signature of a real symmetric matrix is 
defined as follows:
\begin{Definition}[Signatures of symmetric matrices]\label{def:signature}
Let $H$ be a real symmetric matrix. 
Let $N_+(H)$, $N_-(H)$, and $N_0(H)$ be the numbers of positive, negative, and zero eigenvalues of $H$, respectively. 
We call the triplet  
$$
 N(H) := \(N_+(H), N_-(H), N_0(H)\)
$$  
the  {\bf signature} of  $H$. 
\end{Definition}

We  provide here a formula for computing the signature of the Hessian matrix of the potential function of the formation control system. Note that from Lemma~\ref{LEQUIV}, the signature of $H_{p'}$ is invariant as long as $p' \in O_p$. 
Also, note that in terms of the signature,  a critical orbit  $O_p$ (which is of dimension~$3$)  
is {exponentially stable} if and only if 
$
 N(H_p) = (2n -3,0,  3)
$. 

Let $(G, p)$ be a framework, with $G$ a triangulated Laman graph and $p\in P_G$. Let $\left \{ \(G_i,p_i\) \right \}^m_{i=1}$, with $G_i = (V_i, E_i)$,  be the frameworks  associated with the independent partition for $(G,p)$. Let $\cal{S}_i$ be the formation subsystem induced by the subgraph $G_i$. We recall that each $\cal{S}_i$ is a gradient system. Specifically, let $V_i = \{i_1,\ldots, i_{|V_i|}\}$. Then, the associated potential function $\Phi_i$ is given by
\begin{equation}\label{eq:inducedpotential}
\Phi_i(x_{i_1},\ldots, x_{i_{|V_i|}}) = \sum_{(i_j,i_k)\in E_i}  \,   \int^{\|x_{i_j} - x_{i_k}\|}_{1} sf_{i_j i_k}(s) ds. 
\end{equation}
Let $H_{p_i}$ be the Hessian of $\Phi_i$ at $p_i$, and  
$
N(H_{p_i})$ be the signature of $H_{p_i}$. 
With the definitions and notations above, we state the following result:

\begin{theorem}\label{MBIF}
Let $(G,p)$ be a framework, with $G$ a triangulated Laman graph and $p\in P_G$. 
Let $\{(G_i,p_i)\}^m_{i=1}$  be the frameworks associated with the independent partition for $(G,p)$. Let $H_p$ be the Hessian of $\Phi$ at $p$, and $H_{p_i}$, for $i = 1,\ldots, m$, be the Hessian of $\Phi_i$ at $p_i$. Then, 
\begin{equation}\label{INDEXF}
\left\{
\begin{array}{l}
N_+(H_p)=\sum^m_{i=1}N_+(H_{p_i}), \vspace{3pt}\\
N_-(H_p)=\sum^m_{i=1}N_-(H_{p_i}).
\end{array}
\right.
\end{equation}\, 
\end{theorem}

\begin{Remark}\normalfont
We note here that the configuration $p$ in Theorem~\ref{MBIF} does {\it not} need to be an equilibrium of system~\eqref{MODEL}. The formula~\eqref{INDEXF} holds for all $p\in P_G$. Furthermore, the control laws $u_{ij}$'s that define the potential $\Phi$ do not need to be monotone att./rep. functions, nor the potential $\Phi$ needs to be an equivariant Morse function. Indeed, Theorem~\ref{MBIF} holds as long as $G$ is a triangulated Laman graph and the control laws are continuously differentiable.   $\qed$
\end{Remark}

Note that $\{E_i\}^m_{i=1}$ is a partition of $E$, and hence from~\eqref{eq:inducedpotential}, we have $\Phi(p) = \sum_{i=1}^m \Phi_i(p_i)$. The formula~\eqref{INDEXF} would hold trivially if each $ \Phi_i$ were {\it independent} of $\Phi_{j}$ for $j\neq i$, i.e., the sets of variables of $\Phi_i$'s are mutually distinct. For example, if $\Phi(x_1,\ldots, x _n) = \sum^N_{i=1}\Phi_i(x_i) $, then the resulting gradient system is fully decoupled, and hence Theorem~\ref{MBIF} follows immediately from the fact that the Hessian of $\Phi$ is block-diagonal. However, such a decoupling argument does \emph{not} apply here. Indeed,  consider the following example: let $G$ be a complete graph of three vertices and $p\in P_G$ be a non-degenerate triangle. The independent partition for $(G,p)$ is then given by
$$\left\{E_1=\{(2,3)\}, E_2=\{(1,3)\}, E_3=\{(1,2)\}\right\}.$$ 
The corresponding potential functions $\Phi_i$, for $i = 1, 2, 3$,  are given by  $$\Phi_i(x_j,x_k)= \int_1^{\|x_j-x_k\|} s f_{jk}(s) ds$$ for $i,j,k$ distinct integers in $\{1,2,3\}$. These functions are clearly not independent of others, and hence the Hessian of $\Phi$ is not block-diagonal. Nevertheless, Theorem~\ref{MBIF} shows that formula~\eqref{INDEXF} still holds.

\paragraph{Analysis and proof of Theorem~\ref{MBIF}}  
We first have some definitions and notations.  
Let $N_{p_i}O_{p_i}$ be the normal space of $O_{p_i} $ at $p_i$ in $P_{G_i}$, i.e., $N_{p_i}O_{p_i}$ is the subspace of $\R^{2|V_i|}$ normal to the tangent space $T_{p_i}O_{p_i}$. Note that from Lemma~\ref{LEQUIV},  $T_{p_i}O_{p_i}$ is in the kernel of $H_{p_i}$. Since $H_{p_i}$ is a symmetric matrix, we have that $N _{p_i}O_{p_i}$ is invariant under $H_{p_i}$, i.e., if $v\in N_{p_i}O_{p_i}$, then $H_{p_i}v \in N_{p_i}O_{p_i}$. We also note that $ \dim T_{p_i}O_{p_i} = 3$, and hence by the fact that $|E_i| = 2|V_i| - 3$, we obtain 
$
\dim N_{p_i} O_{p_i}  = |E_i|
$.  

Now, for each $i=1,\ldots, m$,  choose an orthonormal basis  
$\{v_{i_1}, \ldots ,v_{i_{|E_i|}}\}$ of $N _{p_i}O_{p_i}$ such that each $v_{i_j}$ is an eigenvector of $H_{p_i}$, with $\lambda_{i_j}$ the corresponding eigenvalue. 
Recall that  $D_{p_i}\eta_i: T_{p_i}\cal{U}_i \longrightarrow T_pP_G \approx \R^{2n}$ is the derivative map defined in Corollary~\ref{cor:derivativeofeta}. 
With the eigenvectors $v_{i_j}$, for $j = 1,\ldots, |E_i|$, and the derivative map $D_{p_i}\eta_i$,  we define a set of vectors in $\R^{2n}$ as follows:
\begin{equation}\label{eq:definitionofwij}
w_{i_j} := D_{p_i}\eta_i(v_{i_j}), \hspace{10pt} \forall\, j = 1,\ldots,|E_i|.     
\end{equation}
Then, for each $i = 1,\ldots, m$, we define two matrices  as follows: 
$$
\left\{
\begin{array}{l}
\Lambda_i := \diag(\lambda_{i_1},\ldots, \lambda_{i_{|E_i|}}) \in \R^{|E_i|\times |E_i|}, \vspace{3pt}\\
W_i := (w_{i_1}, \ldots, w_{i_{|E_i|}}) \in \R^{2n \times |E_i|}.  
\end{array}
\right. 
$$
Further, we let $\{w_{0_1}, w_{0_2}, w_{0_3}\}$ be an orthonormal basis of $T_{p}O_{p}$, and define  $$W_0 : = \( w_{0_1}, w_{0_2}, w_{0_3} \) \in \R^{2n\times 3}.$$
By the fact that 
$\sum^m_{i=1} |E_i| + 3 = |E| + 3 = 2n$, 
we thus obtain two $2 n\times2 n$ square matrices as follows:
$$
\left\{
\begin{array}{l}
\Lambda : = \diag\(0_{3\times 3}, \Lambda_1, \ldots, \Lambda_m\), \vspace{3pt} \\
W : = \( W_{0}, W_{1}, \ldots, W_{m} \). 
\end{array}
\right. 
$$
In particular, $\Lambda$ is a diagonal matrix. Its diagonal entries are comprised of three zeros and the eigenvalues of $H_{p_i}$ for all $i = 1,\ldots, m$.   
With the notations above, we state the following result: 

\begin{pro}\label{pro:diagonalizationofH}
Let $\Lambda$ and $W$ be the two square matrices defined above. Then, $W$ is nonsingular, and moreover, 
$W^\top H_p W = \Lambda$.  
\end{pro}

We refer to Appendix~A.2 for a proof of Proposition~\ref{pro:diagonalizationofH}. Theorem~\ref{MBIF} is then established by appealing to Proposition~\ref{pro:diagonalizationofH} and Sylvester's Law of inertia~\cite{sylvester1852xix}, which says that $N(H_p) = N(\Lambda)$.

\subsection{Proof of Theorem~\ref{MAIN}}
We prove in this subsection Theorem~\ref{MAIN}. Let $p\in P_G$ be a critical point of $\Phi$, and $O_p$ be the corresponding critical orbit. Recall that $O_p$ is said to be nondegenerate if $N_0(H_p) = 3$, and exponentially stable if $N(H_p) = (2n - 3, 0, 3)$. 
We first state a fact as a corollary to Theorem~\ref{MBIF}: 

\begin{cor}\label{cor6}
Let $(G,p)$ be a framework, with $G$ a triangulated Laman graph and $p\in P_G$ a critical point of $\Phi$. Let $\{(G_i,p_i)\}^m_{i=1}$ be the frameworks associated with the independent partition for $(G,p)$. Then,  
\begin{enumerate}
\item The critical orbit $O_p$ is nondegenerate if and only if each $O_{p_i}$ is nondegenerate. 
\item The critical orbit $O_p$ is exponentially stable if and only if each $O_{p_i}$ is exponentially stable. 
 \end{enumerate}\,
\end{cor}

\begin{proof}
For part 1, we need to show that $N_0(H_p) = 3$ if and only if  $N_0(H_{p_i}) = 3$ for all $i=1,\ldots, m$. Since $G$ is a Laman graph,  we have 
$
|E| = 2|V| - 3 
$. On the other hand,  the dimension of the formation system is $2|V|$, and hence  
$
N_+(H_p) + N_-(H_p) + N_0(H_p) = 2|V|
$. 
Using the fact that  $N_0(H_p) \ge 3$ (from Lemma~\ref{LEQUIV}), we have
$
N_+(H_p) + N_-(H_p) \le |E|
$, 
and the equality holds if and only if  $N_0(H_p) = 3$. 
Similarly, since each $G_i$ is also a Laman graph,  we obtain
\begin{equation}\label{eq:n+n-i}
N_+(H_{p_i}) + N_-(H_{p_i}) \le |E_i|,
\end{equation}
and the equality holds if and only if  $N_0(H_{p_i}) = 3$. Then, from Theorem~\ref{MBIF}, we have
\begin{equation}\label{eq:sumn+n-}
N_+(H_p) + N_-(H_p)   \le  \sum^m_{i=1} | E_i | = |E|.
\end{equation}
The last equality holds because $\{E_i\}^m_{i=1}$ is a partition of $E$. Following~\eqref{eq:n+n-i} and~\eqref{eq:sumn+n-}, we see that $N_+(H_p) + N_-(H_p) = |E|$ if and only if  $N_+(H_{p_i}) + N_-(H_{p_i}) = |E_i|$ for all $i = 1,\ldots, m$.  
In other words, $N_0(H_p) = 3$ if and only if  $N_0(H_{p_i}) = 3$ for all $i =1, \ldots, m$. For part~2,  it suffices to show that $N_-(H_p) = 0$ if and only if $N_-(H_{p_i}) = 0$ for all $i = 1,\ldots, m$, which is an immediate consequence of~\eqref{INDEXF}.
\qquad\end{proof}

With the use of Theorem~\ref{MBIF} and Corollary \ref{cor6}, establishing Theorem~\ref{MAIN} can be reduced to determining the  signature of $H_p$ for $p$ either a strongly rigid configuration, or a line configuration. 

\paragraph{Critical configurations that are strongly rigid} We deal with the first case in the following corollary to Theorem~\ref{MBIF}:

\begin{cor}\label{cor:forstronglyrigidconf}
Let system~\eqref{MODEL} be a triangulated formation system. Let $p$ be a critical orbit of $\Phi$. If $O_p$ is strongly rigid, then it is exponentially stable, and moreover, the distance between $ x_i$ and $ x_j$ in $p$ is the target distance $\ol d_{ij}$ for all $(i,j)\in E$. 
\end{cor}

\begin{proof}
Let $\{(G_i,p_i)\}^{m}_{i=1}$ be the frameworks associated with the independent partition for $(G,p)$.  Since $p$ is strongly rigid, the independent partition for $(G,p)$ has each edge of $G$ in a distinct subset, i.e. each subgraph $G_i = (V_i, E_i)$ has only two vertices and one edge and we  have $m = 2n - 3$. Also,  each $p_i$ is a sub-configuration of two agents, and  by Proposition~\ref{pro:equilibriumcondition}, it is an equilibrium of the subsystem induced by $G_i$. Hence, for all $(i,j) \in E$, we have $f_{ij}(d_{ij})=0$, which implies that $d_{ij} = \ol d_{ij}$.

 We now compute the signature of $H_{p_i}$. Suppose that $p_i$ consists of agents $ x_{j}$ and $ x_{k}$; then, the potential function $\Phi_i$ induced by $G_i$ is given by
 $$
 \Phi_i(x_j, x_k) = \int^{\|x_j - x_k \|}_{1} s f_{jk}(s) ds.  
 $$
 From Lemma~\ref{LEQUIV}, we may rotate and/or translate $p$ so that both $ x_{j}$ and $ x_{k}$ are on the first coordinate. 
A direct computation of $H_{p_i}$ yields that 
\begin{equation*}
H_{p_i} = 
\ol d_{jk}\, f'_{jk}\(\ol d_{jk}\)
\begin{pmatrix}
1 & 0 & -1 & 0\\
0 & 0 & 0 & 0\\
-1 & 0 & 1 & 0\\
0 & 0 & 0 & 0
\end{pmatrix}. 
\end{equation*} 
Moreover, from condition~C1 in Definition~\ref{def:interactionfunction}, we have 
\begin{equation*}
\frac{d}{dx}\(xf_{jk}(x) \) \left |_{x = \ol d_{jk}} \right.= \ol d_{jk} \, f'_{jk}\(\ol d_{jk} \) >0,  
\end{equation*}
and hence 
$ N(H_{p_i}) = (1, 0, 3)$. 
Now, appealing to Theorem~\ref{MBIF},  we obtain
\begin{equation*}
\left\{
\begin{array}{l}
N_+(H_p) = \sum^m_{i=1}N_+(H_{p_i}) = m, \vspace{3pt}\\
N_-(H_p) =  \sum^m_{i=1}N_-(H_{p_i}) = 0.
\end{array}
\right.
\end{equation*} 
Using the fact that $m = 2n - 3$, we conclude that 
$
 N(H_p) = (2n-3,0,3)
$,  
which implies that $O_p$ is exponentially stable. 
\qquad\end{proof}

\paragraph{Critical line configurations} We now  focus on the case where $p\in P_{G}$ is a critical line configuration. 
 It is computationally convenient to stack the first-coordinates (resp. the second-coordinates) of the agents in vectors $a$ and $b$. Precisely, we re-arrange the entries of $p=(x_1,\ldots,x _n)$ as follows: write $x_i = (a_i, b_i)$, for $a_i$ and $b_i$ scalars, and define
$a:= (a_1,\ldots, a _n)$  and $b: = \(b_1, \ldots, b_n \)$. 
We then redefine the vector $p$ as
$
p := \(a, b\) \in \R^{2n}
$.

An advantage of the re-arrangement is that the dynamics~\eqref{MODEL} can be re-cast into a matrix form. Specifically, for a configuration $p\in P_G$, let $F_p$ be an $n \times n $ symmetric, zero-row/column-sum matrix. The off-diagonal entries of $F_p$ are given by
\begin{equation}\label{GMAT'}
F_{p, ij}:=
\left\{
\begin{array}{ll}
-f_{ij}(d_{ij}) & \text{if }(i,j)\in E, \\
0 & \text{otherwise}. 
\end{array}
\right.
\end{equation}
The diagonal entries are then obtained by using the condition that the rows/columns  of $F_p$ sum to zero. So then, with the re-arrangement, system~\eqref{MODEL} can be expressed as follows:
$$
\begin{pmatrix}
\dot a \\
\dot b
\end{pmatrix} = -
\begin{pmatrix}
F_p & 0\\
0 & F_p
\end{pmatrix}
\begin{pmatrix}
a \\
b
\end{pmatrix}
$$

We compute below the Hessian matrix $H_p$. 
By Lemma~\ref{LEQUIV}, we can assume, without loss of generality, that the line configuration $p$ is \emph{on the first coordinate}, i.e., we assume that $ b =  0$. So then, by computation, the Hessian $H_p$ is a block-diagonal matrix given by
\begin{equation}\label{HESS}
H_p =
\begin{pmatrix}
D_p & 0\\
0 & F_p
\end{pmatrix}
\end{equation}
where $D_p$ is again an $n\times n $ symmetric, zero-row/column-sum matrix. The off-diagonal entries of $D_p$ are given by  
\begin{equation}\label{GMAT}
D_{p, ij}:=
\left\{
\begin{array}{ll}
\left. -\dfrac{d}{dx}\right|_{x = d_{ij}} \(xf_{ij}(x)\)   & \text{if }(i,j)\in E, \\
0 & \text{otherwise}.
\end{array}
\right.
\end{equation}
The diagonal entries of $D_p$ are again determined by the condition that rows/columns of $D_p$ sum to zero. 
We now show that if $n \ge 3$, then $N_-(H_p) \ge 1$. First, note that from~\eqref{HESS}, we have
$
N_-(H_p) = N_-(D_p) + N_-(F_p)
$. 
For the matrix $D_p$, we note that from~\eqref{GMAT} and condition~C1 in Definition~\ref{def:interactionfunction}, all off-diagonal entries of $D_p$ are non-positive. Hence, $D_p$ is the negative of an infinitesimally stochastic matrix; in particular, by the Gershgorin circle theorem, all eigenvalues of $D_p$ are non-negative, which implies that $N_-(p) = 0$.  For the matrix $F_p$, we have the following fact: 

\begin{pro}\label{pro:signatureofFp} Let system~\eqref{MODEL} be a triangulated formation system of at least three agents. Let $(G,p)$ be a line framework, with $p\in P_G$ a critical point of $\Phi$ on the first coordinate. Then, $N_-(F_p) \ge 1$, for $F_p$ defined by~\eqref{GMAT'}.  
\end{pro}

We refer to Appendix~A.3 for a proof of Proposition~\ref{pro:signatureofFp}. 
We also refer to~\cite{sun2015rigid} for a similar result about instability of degenerate critical formations.  
Equipped with the results above, we now prove Theorem~\ref{MAIN}. 

{\em Proof of Theorem~\ref{MAIN}}.
Let $p$ be an equilibrium of system \eqref{MODEL}. From Corollary~\ref{cor:forstronglyrigidconf}, if $p$ is strongly rigid, then $O_p$ is exponentially stable.  We now  assume that $p$ is not strongly rigid, and show that $O_p$ is unstable.  
Let $\{(G_i,p_i)\}^m_{i=1}$ be the frameworks  associated with the independent partition for $(G,p)$; without loss of generality, we assume that  $p_1$ contains at least three agents. 
From Proposition~\ref{pro:signatureofFp}, we have $N_-(H_{p_1}) \ge 1$. Appealing to Theorem~\ref{MBIF}, we obtain 
$
N_-(H_p) \ge N_-(H_{p_1})\ge 1
$, 
which implies that the orbit $O_p$ is unstable. We have thus proved that  a critical orbit of system~\eqref{MODEL} is stable if and only if  it is strongly rigid. From Corollary~\ref{cor:forstronglyrigidconf}, the set of stable critical orbits is characterized by the condition that $d_{ij}=\ol d_{ij}$ for all $(i,j)\in E$, and hence there are as many as $2^{n-2}$ stable critical orbits  in total. The convergence of system \eqref{MODEL} follows from Lemma~\ref{CONVGG}.
\qquad\endproof

\section{Future work and conclusions}\label{sec:conclusion}

To conclude, designing control laws that stabilize only the target configurations of a formation is known to be a challenging problem. Indeed, the conjunction of the decentralization constraints and the nonlinear nature of the dynamics lead to the appearance of undesirable equilibria in the system. Counting these equilibria is  in general a difficult and open problem, let alone characterizing them. Some progress in this direction has been made (see, for example, \cite{baillieul2007combinatorial,BDO2014CT,UH2013E}), yet no complete characterization is known yet. In this paper,  we have provided a partial solution to the problem by exhibiting a class of undirected graphs and control laws for which only desired configurations are stable.  We have furthermore derived a formula~\eqref{INDEXF} in Theorem~\ref{MBIF}, evaluating the signatures of the Hessians at critical orbits, which may be of independent interest. 

Amongst the topics not addressed in this paper, but of clear practical interest, we single out  imprecision or noise in the distance measurements.  Recall that the control laws $u_{ij}(\cdot, \ol d_{ij})$, for $(i,j)\in E$, takes the relative distance $\|x_j(t) - x_i(t)\|$ as information for its feedback. However, the measurement of the relative distance, taken by agents $x_i$ and $x_j$, may be corrupted by noise. Taking this into account, we propose the following model: 
\begin{equation}\label{eq:stochasticpath}
\dot x_i = \sum_{j\in \cal{N}_i} u_{ij}(\|x_j(t) - x_i(t)\| +\delta_{ij}(t),\ol d_{ij}) (x_j - x_i), \hspace{10pt} \forall i\in V, 
\end{equation} 
where each $\delta_{ij}(t)$  is a stochastic process modeling the measurement noise, or a fixed bias of the sensor.  Other sources of uncertainties are in bearing measurements, which entails studying a model where the terms $(x_j-x_i)$ are noisy or have a fixed bias. The effects of the first source of measurement errors on the dynamics has been studied recently in~\cite{mou2014CDC,USZB}. However, a complete understanding of the effect of measurement errors, and the robustness of the dynamics, is still lacking.
\bibliographystyle{unsrt}
\bibliography{FC}

\Appendix\section*{}
We provide here proofs of Propositions~\ref{pro:srcopendense},~\ref{pro:diagonalizationofH} and~\ref{pro:signatureofFp}. 

\subsection{Proof of Proposition~\ref{pro:srcopendense}}
 Let $Q_G\subset P_G$ be the set of strongly rigid configurations in $P_G$. We first show that $Q_G$ is open and dense, and then show that each configuration $p$ in $Q_G$ is infinitesimally rigid. 
Let vertices $i$, $j$, and $k$ form a $3$-cycle of $G$; we then let $T_{(i, j, k)}$ be a proper subspace of $P_G$ as follows:
$$
T_{(i,j,k)} := \left \{ p\in P_G \mid \det\(x_j - x_i, x_k - x_i \) = 0 \right \}.
$$
The codimension of  $T_{(i,j,k)}$ in $P_G $ is one. Further, we define
$
T_G := \cup_{(i,j,k)} T_{(i,j,k)} 
$   
 where the union is taken over all triplets of vertices $(i,j,k)$ such that they form a $3$-cycle of $G$.  Then,  
 $
 Q_G = P_G - T_G
$   
which implies that $Q_G$ is an open dense subset of $P_G$.

Recall that for a graph $G$, the distance map $\rho_G: P_{G}\longrightarrow \R^{|E|}_+$  is defined as
$$
\rho_G(p) = \(\cdots, \| x_i -  x_j\|^2,\cdots\)_{(i,j)\in E}. 
$$
Let $p$ be in $Q_G$; we now show that 
$$
\rk\(\frac{\partial \rho_G(p)}{\partial p}\) = 2n - 3.
$$
The proof will be carried out by induction on the number of vertices of $G$.   For the base case $n = 2$, we have $\rho_G(x_1, x_2)= \|x_2 -x_1\|^2$, and hence 
$$
\frac{\partial \rho_G(p)}{\partial p} = \( x^\top_1 - x^\top_2,\, x^\top_2 - x^\top_1 \).
$$
Since $x_1\neq x_2$, the rank of $\partial \rho_G(p) / \partial p$ is one. 

For the inductive step, assuming that the statement holds for $n < k$ for $k \ge 3$, we prove it for $n = k$. Choose a Henneberg sequence of $G$, and label the vertices of $G$ such that vertex $1$ is the last vertex appearing in the sequence, linking to the vertices~$2$ and~$3$. Let $G^* = \(V^*, E^*\)$ be the  subgraph induced by the vertices $V^* :=\{2, \ldots, k\}$, and $(G^*, p^*)$ be the corresponding framework. Note that $p^*$ is strongly rigid. Hence, from the induction hypothesis, we have
$$
\rk\(\left. \frac{\partial \rho_{G^*}(p')}{\partial p'} \right |_{p' = p^*}\) = 2(k-1) - 3.
$$ 
On the other hand, by computation, we have
$$
\frac{\partial \rho_G(p)}{\partial p} = 
\begin{pmatrix}
A_{11} & A_{12}\\
0 &  \left. \frac{\partial \rho_{G^*}(p')}{\partial p'} \right |_{p' = p^*}
\end{pmatrix}
$$
with $A_{11}$ a $2\times2$ matrix given by $A_{11} = \( x_1 - x_2, x_1 -x_3 \)^\top$. 
Since $p$ is strongly rigid, we have that $(x_1 - x_2)$ and $(x_1 - x_3)$ are linearly independent. So then,
$$
\rk\(\frac{\partial \rho_G(p)}{\partial p}\) = \rk\(A_{11} \) + \rk\(  \left. \frac{\partial \rho_{G^*}(p')}{\partial p'} \right |_{p' = p^*}  \) = 2k - 3.
$$
This completes the proof. 
\qquad \endproof


\subsection{Proof of Proposition~\ref{pro:diagonalizationofH}}
We first show that $W$ is nonsingular. Choose coefficients $\alpha_{i_j}$ in $\R$, and let 
\begin{equation}\label{eq:testeq}
\sum^m_{i=1}\sum^{|E_i|}_{j=1}\alpha_{i_j}w_{i_j} + \sum^3_{j = 1} \alpha_{0_j} w_{0_j} = 0.
\end{equation}
We show that all the coefficients $\alpha_{i_j}$ are zero. First, picking an $i = 1,\ldots, m$, we show that $\alpha_{i_j} = 0$ for all $j = 1,\ldots,|E_i|$. 
To this end, we evaluate $w_{i'_j} |_{V_i}$ for $i' = 0, 1,\ldots, m$. There are three cases to consider:
\begin{enumerate}
\item Suppose that $i' = 0$; since 
$w_{0_j}\in T_{p}O_p$, we have
\begin{equation}\label{eq:w0j}
w_{0_j} |_{V_i} \in T_{p_i}O_{p_i}, \hspace{10pt} \forall \, j = 1, 2, 3. 
\end{equation}
\item Suppose that $i' > 0$ and $i' \neq i$; then, from Corollary~\ref{cor:derivativeofeta}, we obtain
\begin{equation}\label{eq:wi'j}
w_{i'_j} |_{V_{i}} \in T_{p_i}O_{p_i}, \hspace{10pt} \forall \, j  = 1,\ldots,|E_{i'}|. 
\end{equation}
\item Suppose that $i' = i$; then, from Corollary~\ref{cor:derivativeofeta}, we obtain
\begin{equation}\label{eq:wij}
w_{i_j} |_{V_i} = v_{i_j}, \hspace{10pt} \forall \, j  = 1,\ldots,|E_{i}|.
\end{equation}
\end{enumerate}
Combining~\eqref{eq:w0j} and~\eqref{eq:wi'j}, we have
$$
v_{i_0} := (\sum^m_{i' = 1,\\ i'\neq i}\sum^{|E_{i'}|}_{j = 1}\alpha_{i'_j}\, w_{i'j} + \sum^3_{j = 1} \alpha_{0_j}\, w_{0_j}) \mid_{V_i} \in  T_{p_i}O_{p_i}, 
$$ 
Combining this with~\eqref{eq:wij}, we have
\begin{equation}\label{eq:waitingforali}
(\sum^m_{i=1}\sum^{|E_i|}_{j=1}\alpha_{i_j}w_{i_j} + \sum^3_{j = 1} \alpha_{0_j} w_{0_j} ) \mid_{V_i} = \sum^{|E_i|}_{j = 1} \alpha_{i_j}v_{i_j} + v_{i_0} = 0.
\end{equation}
Note that the vectors $v_{i_j}$, for $j = 1,\ldots, |E_i|$, form an orthonormal basis of $N _{p_i}O_{p_i}$, and hence~\eqref{eq:waitingforali} holds if and only if $v_{i_0} = 0$ and $\alpha_{i_j} = 0$ for all $j = 1,\ldots, |E_i|$. It now remains to show that $\alpha_{0_j} = 0$ for all $j = 1, 2, 3$. Note that~\eqref{eq:testeq} is now reduced to
$
\sum^3_{j = 1} \alpha_{0_j} w_{0_j} = 0 
$.   
Since the three vectors $w_{0_1}, w_{0_2}, w_{0_3}$ form a basis of $T_pO_p$, we conclude that $\alpha_{0_j} = 0$ for all $j = 1,2,3$. We have thus proved that the vectors $w_{i_j}$ are linearly independent, and hence $W$ is nonsingular.  


We now show that $W^\top H_{p} W = \Lambda$. 
First, recall that $\Phi_i$ is the potential function induced by $G_i = (V_i, E_i)$, which is defined over $P_{G_i}$. Let $p$ be in $P_{G}$; we define $\wh \Phi_i: P_G\longrightarrow \R$  as follows:
$$
\wh \Phi_i(p) : = \Phi_i(p_i)= \sum_{(j,k)\in E_i} \int^{d_{jk}}_{1} sf_{jk}(s) ds. 
$$
Since $\{E_i\}^m_{i=1}$ is a partition of $E$, we have
$
\Phi(p) = \sum^m_{i=1}\wh \Phi_i(p)  
$.  Hence, if we define $2n\times 2n$ matrices 
$
\wh H_{p_i}:= \partial^2\wh \Phi_i(p)/\partial p^2$  for $i = 1,\ldots, m$,    
then   
$
H_{p} = \sum^m_{i=1} \wh H_{p_i}
$. 
On the other hand, we recall that $H_{p_i}$ is the Hessian of $\Phi_i$ at $p_i$, which is a $|V_i| \times |V_i|$ matrix.  
A relationship between $\wh H_{p_i}$ and $H_{p_i}$ is the following: Each $\wh H_{p_i}$ can be derived by adding zero-rows/columns into $H_{p_i}$. Indeed, from definition,  if $j\notin V_i$, then the $j$-th row/column of $\wh H_{p_i}$ is zero. Moreover, by removing these zero-rows/column out of $\wh H_{p_i}$, we obtain $H_{p_i}$. We now use the relationship to derive some relevant facts. 

We recall that for a fixed $i > 0$, the vectors $v_{i_j}$, for $j = 1,\ldots, |E_i|$, form the orthonormal basis of $N _{p_i}O_{p_i}$. Moreover, they are eigenvectors of $H_{p_i}$, i.e.,  
$
H_{p_i} v_{i_j} = \lambda_{i_j} v_{i_j}
$. We also recall that from Lemma~\ref{LEQUIV}, $T_{p_i}O_{p_i}$ is in the kernel of $H_{p_i}$. 
Then, by combining~\eqref{eq:w0j}, ~\eqref{eq:wi'j}, and~\eqref{eq:wij}, we obtain that for $i' = 1,\ldots,m$, 
\begin{equation}\label{eq:1:04pmdec012015}
H_{p_i} w_{i'_j} |_{V_i} = \delta_{i' i}\, \lambda_{i_j} \, v_{i_j}, \hspace{10pt} \forall\, j = 1, \ldots, |E_i|,
\end{equation}
where $\delta_{i'i}$ is the Kronecker delta. 
Now, for each vector $v_{i_j} \in \R^{2|V_i|}$, we define a vector $\wh v_{i_j} \in \R^{2n}$ by adding zero entries into $v_{i_j}$: the zero entries are added in a way such that 
$
\wh v_{i_j} |_{V_i} = v_{i_j} 
$. Following~\eqref{eq:1:04pmdec012015} and the relationship between $\wh H_{p_i}$ and $H_{p_i}$, we obtain
$$
\wh H_{p_i} \, w_{i'_j} = \delta_{i'i} \, \lambda_{i_j}\, \wh v_{i_j}, \hspace{10pt} \forall\, j = 1, \ldots, |E_i|.
$$
Further, by the fact that $H_{p} = \sum^m_{i=1} \wh H_{p_i}$,   we have
\begin{equation}\label{eq:thestepbeforethelaststep}
H_p w_{i_j} = 
\left\{
\begin{array}{ll}
\lambda_{i_j} \, \wh v_{i_j} & \mbox{if } i \ge 1, \\
0  & \mbox{if } i = 0. 
\end{array}\right. 
\end{equation}
To compute $W^\top H_p W$,  we now evaluate $\langle w_{i'_{j'}}, \, H_p w_{i_j} \rangle$. First, note that if $i > 0$, then
\begin{equation}\label{eq:thelaststepreally}
\langle w_{i'_{j'}}, \,  \wh v_{i_j}  \rangle =   \langle w_{i'_{j'}} |_{V_{i}}, \, v_{i_j}  \rangle = \delta_{i'i} \langle v_{i_{j'}}, \, v_{i_j} \rangle = \delta_{i'i}\delta_{j'j}.  
\end{equation}
The first equality holds because $\wh v_{i_j}$ is defined by adding zero entries to $v_{i_j}$; the second equality holds because $v_{i_j}\in N_{p_i}O_{p_i}$ and 
by~\eqref{eq:w0j},~\eqref{eq:wi'j}, and~\eqref{eq:wij}, 
$$
w_{i'_{j'}} |_{V_{i}} = 
\left\{
\begin{array}{ll}
v_{i_{j'}} & \mbox{if } i' = i, \vspace{3pt}\\
\in T_{p_i}O_{p_i} &  \mbox{if } i' \neq i;
\end{array}
\right.
$$ 
The last equality holds because $v_{i_1},\ldots, v_{i_{|E_i|}}$ are orthonormal. 
Then, combining~\eqref{eq:thestepbeforethelaststep} with~\eqref{eq:thelaststepreally}, we obtain
$$
\langle w_{i'_{j'}}, \, H_p w_{i_j} \rangle = \lambda_{i_j} \delta_{i'i} \delta_{j'j},
$$
which implies that $W^\top H_p W = \Lambda$. 
\qquad \endproof


\subsection{Proof of Proposition~\ref{pro:signatureofFp}}
The proof will be carried out by induction on the number of vertices of~$G$. 
For the base case $n = 3$, we assume, without loss of generality, that agent $x_1$ lies in between $x_2$ and $x_3$, and $a_2 < a_1 < a_3$. First, we show that 
\begin{equation}\label{eq:triangleineq1}
d_{12} = a_{1} - a_2 < \ol d_{12} \hspace{5pt} \mbox{ and } \hspace{5pt} d_{13} = a_{3} - a_1 < \ol d_{13}.
\end{equation}
Suppose that, to the contrary,~\eqref{eq:triangleineq1} does not hold; without loss of generality, we assume that $d_{12} \ge \ol d_{12}$. Note that $p$ is an equilibrium, and hence the dynamics of $x_1$ is zero at $p$:
\begin{equation}\label{eq:dota1equalzero}
 \dot a_1 = f_{12}(d_{12})(a_2 - a_1) + f_{13}(d_{13}) (a_3 - a_1) = 0.
\end{equation}
Since $d_{12}  \ge \ol d_{12}$, from condition C1 in Definition~\ref{def:interactionfunction}, we have $f_{12}(d_{12}) \ge 0$. This, in particular, implies that
$$
f_{13}(d_{13}) = \frac{a_1 -a_2}{a_3 - a_1} f_{12}(d_{12}) \ge 0.
$$
Again, by condition C1 in Definition~\ref{def:interactionfunction}, we have $d_{13} \ge \ol d_{13}$. Then, by the strict triangle inequalities associated with $G$, we obtain
$$
d_{23} = d_{12} + d_{13} \ge \ol d_{12} + \ol d_{13} > \ol d_{23},
$$
and hence, $f_{23}(d_{23}) > 0$. But then, the dynamics of $x_2$ is nonzero at $p$; indeed, we have
$$
 \dot a_2 = f_{12}(d_{12}) (a_1 - a_2) + f_{23}(d_{23}) (a_3 - a_2) > 0,
$$
which is a contradiction. We have thus established~\eqref{eq:triangleineq1}. Following~\eqref{eq:triangleineq1}, we have
$$
F_{p, 11} = f_{12}(d_{12}) + f_{13}(d_{13}) < 0,
$$
which implies that $N_-(F_p) \ge 1$. 

For the inductive step, assuming that Proposition~\ref{pro:signatureofFp} holds for $n < k$, we prove it for $n = k$. Since $G$ is a triangulated Laman graph, there exists a vertex, say vertex~$k$, of degree~$2$. Without loss of generality, we assume that vertex~$k$ is adjacent to vertices~$i$ and~$j$. There are two cases to consider: 

{\it Case I}. Suppose that $x_k$ lies in between $x_i$ and $x_j$; without loss of generality, we can assume that 
\begin{equation}\label{eq:aiakaj}
a_i < a_k < a_j.
\end{equation}
 Since $p$ is an equilibrium, the dynamics of $x_k$ is zero at $p$, i.e., 
\begin{equation}\label{eq:akzerofirst}
\dot a_k = f_{ik}(d_{ik})(a_i - a_k) + f_{jk}(d_{jk}) (a_j - a_k) = 0.
\end{equation}
We may assume that 
\begin{equation}\label{eq:dijoldij}
d_{ik} \ge \ol d_{ik} \hspace{5pt} \mbox{ and } \hspace{5pt} d_{jk} \ge \ol d_{jk}.
\end{equation}
because otherwise, say $d_{ik} < \ol d_{ik}$, then we have $f_{ik}(d_{ik}) < 0$, and hence from~\eqref{eq:akzerofirst}, we obtain $f_{jk}(d_{jk}) < 0$, which implies that  $$F_{p,kk} = f_{ik}(d_{ik}) + f_{jk}(d_{jk})< 0,$$ and hence $N_-(F_p) \ge 1$. 
We thus assume that~\eqref{eq:dijoldij} holds, and state below some implications of it. First, by combining~\eqref{eq:aiakaj},~\eqref{eq:akzerofirst}, and~\eqref{eq:dijoldij}, we obtain 
\begin{equation}\label{eq:d12f12d13f13}
f_{ik}(d_{ik})(a_k -a_i) =  f_{jk}(d_{jk}) (a_j -a_k) \ge 0.
\end{equation}
Also, note that by the strict triangle inequalities associated with $G$, we have
\begin{equation}\label{eq:d23greaterthanold23}
d_{ij} = d_{ik} + d_{ik} \ge \ol d_{ik} + \ol d_{jk} > \ol d_{ij},
\end{equation}
which implies that $f_{ij}(d_{ij}) > 0$. 

Let $G^*= \(V^*, E^*\)$ be the subgraph of $G$ induced by the vertices $V^* := \{1,\ldots, k - 1\}$, and $(G^*, p^*)$ be the corresponding framework.
We recall that  a reduction of system~\eqref{MODEL}, denoted by $\cal{R}$, for $(G, p)$ is defined as follows:  First,  choose a function $g_{ij}$ in $\operatorname{C}^1(\R_+,\R)$ such that the value of $g_{ij}$ at $d_{ij}$ satisfies the following condition:
\begin{equation}\label{eq:recallg23f12}
g_{ij}(d_{ij})(a_j - a_i) = f_{ik}(d_{ik})(a_k - a_i) = f_{jk}(d_{jk}) (a_j -a_k).
\end{equation}
We then obtain $\cal{R}$, from system~\eqref{MODEL}, by first removing agent $x_k$ out of~\eqref{MODEL}, and then setting the control laws $f^*_{i'j'}$, for $(i',j') \in E^*$, as follows:
\begin{equation*}
f^*_{i'j'} = 
\left\{
\begin{array}{ll}
f_{ij} + g_{ij} & \mbox {if } (i',j') = (i,j), \\
f_{i'j'} & \mbox{otherwise}. 
\end{array}
\right.
\end{equation*}
Furthermore, note that by following~\eqref{eq:dijoldij} and its implications, one can choose $g_{ij}$ such that 
$dg_{ij}(d)$ is non-decreasing in~$d$, with $g_{ij}(\ol d_{ij}) = 0$; indeed, from~\eqref{eq:d23greaterthanold23}, we have $d_{ij}> \ol d_{ij}$, and from~\eqref{eq:d12f12d13f13} and~\eqref{eq:recallg23f12}, we have
$
d_{ij}g_{ij}(d_{ij}) \ge 0
$. 
Then, by this choice of $g_{ij}$, the control law $f^*_{ij}$ is a monotone attraction/repulsion function, 
with $f^*_{ij}(\ol d_{ij}) = 0$. 

We now appeal to the induction hypothesis. First, note that $p$ is an equilibrium of system~\eqref{MODEL}, and hence from Lemma~\ref{lem:systemreduction}, $p^*$ is an equilibrium of $\cal{R}$. Let $F_{p^*} \in\R^{(k-1)\times(k-1)}$ be a symmetric matrix defined as follows: the off-diagonal entries of $F_{p^*}$ are given by
\begin{equation}\label{eq:Fstarp}
F^*_{p, ij}:=
\left\{
\begin{array}{ll}
-f^*_{ij}(d_{ij}) & \text{if }(i,j)\in E^*, \\
0 & \text{otherwise}. 
\end{array}
\right.
\end{equation}
and the diagonal entries of $F_{p^*}$ are determined by the condition that the rows/columns of $F_{p^*}$ sum to zero. 
Note that the formation control system $\cal{R}$ satisfies the assumptions of Theorem~\ref{MAIN}; indeed, all the control laws $f^*_{i'j'}$, for $(i',j')\in E^*$, are in $\mathcal{F}$, with $f^*_{i'j'}(\ol d_{i'j'}) = 0$, and moreover, the target distances $\ol d_{i'j'}$, for $(i',j') \in E^*$, satisfy the strict triangle inequalities associated with $G^*$. Hence, we can apply the induction hypothesis, and obtain $N_-(F_{p^*}) \ge 1$. 

We now relate $N_-(F_{p^*})$ to $N_-(F_p)$. Let $v = \(v_1, \ldots, v_{k-1}\)$ be an eigenvector of $F_{p^*}$ corresponding to an eigenvalue $\lambda$ with $\lambda < 0$. Define a vector $u\in \R^k$ by setting $u:=\(v_1,\ldots, v_k \)$, with the scalar $v_k$ given by
\begin{equation}\label{eq:defv0foru}
v_k := 
\frac{(a_j - a_k)v_{i} + (a_k - a_i)v_{j} }{a_j - a_i}.
\end{equation}
We now show  that 
$
F_p \, u = \lambda \(v, 0\)
$. 
Let $\{r_i\}^k_{n=1}$ and $\{r^*_i\}^{k-1}_{n=1}$ be the row vectors of $F_p$ and of $F_{p^*}$, respectively. Note that if $l \in V - \{i,j,k\}$, then $r_l = (r^*_l, 0)$, and hence $r^\top_l u = r^{*\top}_l v$. For the case $l= i$, we have
$$
r^\top_i u = \sum_{j'\in \cal{N}_i} f_{ij'}(d_{ij'}) (v_i - v_{j'}); 
$$ 
using the facts that $\{j,k\}\subseteq \cal{N}_i$ and
$$
f_{ij}(d_{ij})(v_i - v_j) + f_{ik}(d_{ik})(v_i - v_k) = f^*_{ij}(d_{ij})(v_i - v_j),  
$$ 
we obtain
$$
r^\top_i u = \sum_{j' \in \cal{N}^*_i} f^*_{ij'} (v_i - v_j) = r^{*\top}_i v. 
$$
The same arguments above can be used to prove that $r^\top_j u = r^{*\top}_j v$. 
For the case $l = k$, we need to prove that $r^\top_k u = 0$. First, note that
\begin{equation}\label{eq:r1u}
r^\top_k u =  f_{ik}(d_{ik})(v_k - v_i)+f_{jk}(d_{jk})(v_k - v_j);
\end{equation}
by computation, the right hand side of~\eqref{eq:r1u} yields
$$
\frac{v_i - v_j}{a_j - a_i}\(f_{ik}(d_{ik})(a_i-a_k) + f_{jk}(d_{jk})(a_j - a_k)\), 
$$
which is zero by~\eqref{eq:d12f12d13f13}. 
We have thus shown that $F_p \, u = \lambda( v, 0)$, and hence $$u^\top F_p \, u = \lambda \|v\|^2< 0,$$  
 which implies that $N_-(F_p) \ge 1$.

{\it Case II}. Suppose that $x_k$ does not lie in between $x_i$ and $x_j$; without loss of generality, we assume that $ a_i < a_j < a_k$. The analysis in this case will be similar to the one in the previous case. First, note that the dynamics of $x_k$ is zero, and hence 
$$
f_{ik}(d_{ik}) (a_k - a_i) + f_{jk}(d_{jk}) (a_k - a_j) = 0.  
$$ 
Suppose that $$f_{ik}(d_{ik}) (a_k - a_i) = -f_{jk}(d_{jk}) (a_k - a_j) >  0;$$ 
then, by the fact that  $d_{ik} > d_{jk} > 0$, we have 
$$
F_{p, kk} = f_{ik}(d_{ik}) + f_{jk}(d_{jk}) < 0,
$$
and hence $N_-(F_p) \ge 1$. We thus assume that 
\begin{equation}\label{eq:usefulassumption}
 f_{ik}(d_{ik}) (a_k - a_i) = -f_{jk}(d_{jk}) (a_k - a_j) \le  0
\end{equation}
Some implications of~\eqref{eq:usefulassumption} are stated below. First, note that 
$
 d_{ik} \le \ol d_{ik}$ and $d_{jk} \ge \ol d_{jk}$.  
Also, note that by the strict triangle inequalities, we have
\begin{equation}\label{eq:dijlessthanoldij}
d_{ij} = d_{ik} - d_{jk} \le \ol d_{ik} - \ol d_{jk} < \ol d_{ij}, 
\end{equation}
which implies that $f_{ij}(d_{ij}) < 0$. 

Let $G^* = (V^*, E^*)$ be the subgraph induced by $V^* = \{1,\ldots, k-1\}$, and $(G^*, p^*)$ be the corresponding framework. Similarly, let $\cal{R}$ be a reduction of system~\eqref{MODEL} for $(G,p)$ defined as follows: First, remove $x_k$ out of system~\eqref{MODEL}, and choose a $g_{ij} \in \operatorname{C}^1(\R_+,\R)$ such that 
\begin{equation}\label{eq:crying}
g_{ij}(d_{ij})(a_j - a_i) = f_{ik}(d_{ik})(a_k - a_i) = f_{jk}(d_{jk})(a_j - a_k).
\end{equation} 
Then, set the control laws $f^*_{i'j'}$, for $(i',j') \in E^*$, by
\begin{equation*}
f^*_{i'j'} = 
\left\{
\begin{array}{ll}
f_{ij} + g_{ij} & \mbox {if } (i',j') = (i,j), \\
f_{i'j'} & \mbox{otherwise}. 
\end{array}
\right.
\end{equation*}
We again note that from~\eqref{eq:usefulassumption},~\eqref{eq:dijlessthanoldij}, and~\eqref{eq:crying}, the function $g_{ij}$ can be chosen such that $dg_{ij}(d)$ monotonically increases in~$d$, with $\ol d_{ij} g_{ij}(\ol d_{ij}) = 0$; indeed, from~\eqref{eq:dijlessthanoldij}, we have $d_{ij} < \ol d_{ij}$, and from~\eqref{eq:usefulassumption} and~\eqref{eq:crying}, we have
$
d_{ij} g_{ij}(d_{ij}) \le 0 
$. 
 With the choice of $g_{ij}$, the formation control system $\cal{R}$ satisfies the assumptions of Theorem~\ref{MAIN}. We thus appeal again to the induction hypothesis. Specifically,   
 let $F_{p^*}\in \R^{(k-1)\times (k-1)}$ be defined in the same way as in~\eqref{eq:Fstarp}. Since $p^*$ is an equilibrium of $\cal{R}$ (by Lemma~\ref{lem:systemreduction}), we obtain  $N_-(F_{p^*}) \ge 1$.  Let $v $ be an eigenvector of $F_{p^*}$ corresponding to an  eigenvalue $\lambda$ with $\lambda < 0$. Define $u\in \R^k$ by setting $u := (v, v_k)$, with the scalar $v_k$ defined as in~\eqref{eq:defv0foru}. Then, by the same arguments, we have $F_p u= \lambda(v,0)$, and hence  
$$u^\top F_p \, u = \lambda \|v\|^2 < 0,$$ which implies that $N_-(F_p) \ge 1$. This completes the proof. 
\qquad\endproof

\end{document}